\documentclass[11pt,a4paper]{article}
\usepackage{amscd}
\usepackage{enumerate}

\usepackage{amsmath, amssymb, latexsym}
\usepackage{bbm}
\usepackage{amsfonts}
\usepackage{amsthm}
\usepackage{relsize}
\usepackage{setspace}
\usepackage{geometry}
\usepackage{url}
\usepackage{xspace}
\usepackage{tocloft}
\usepackage{graphics}
\usepackage{graphicx}
\usepackage{lscape}
\usepackage{microtype}
\usepackage{ulem}

\usepackage[usenames, dvipsnames]{color}
\usepackage[utf8]{inputenc}
\usepackage{tikz}

\usepackage[pagebackref=true]{hyperref}
\usepackage[alphabetic]{amsrefs}
\usepackage[english]{babel}

\usepackage{authblk}

\newtheorem{theorem}{Theorem}[section]
\newtheorem{proposition}[theorem]{Proposition}

\newtheorem{corollary}[theorem]{Corollary}
\newtheorem{lemma}[theorem]{Lemma}

\theoremstyle{definition}
\newtheorem{definition}[theorem]{Definition}

\newtheorem{remark}[theorem]{Remark}

\theoremstyle{problem}

\newcommand{\Aut}{\mathrm{Aut}}

\newcommand{\Opp}{\mathrm{Opp}}

\newcommand{\proj}{\mathrm{proj}}
\newcommand{\RR}{\mathbb{R}}
\newcommand{\ZZ}{\mathbf{Z}}

\newcommand{\NN}{\mathbb{N}}

\newcommand{\cat}{$\mathrm{CAT}(0)$\xspace}
\newcommand{\Min}{\mathrm{Min}}
\newcommand{\Ch}{\mathrm{Ch}}
\newcommand{\St}{\mathrm{St}}
\newcommand{\Stab}{\mathrm{Stab}}
\newcommand{\Fix}{\mathrm{Fix}}
\newcommand{\bd}{\partial}

\newcommand{\id}{\operatorname{id}}

\newcommand{\Id}{\operatorname{Id}}

\newcommand{\st}{\operatorname{St}}
\newcommand{\Top}{\operatorname{\mathrm{Top}}}

\def\og{\leavevmode\raise.3ex\hbox{$\scriptscriptstyle\langle\!\langle$~}}
\def\fg{\leavevmode\raise.3ex\hbox{~$\!\scriptscriptstyle\,\rangle\!\rangle$}}

\title{Dynamics of strongly $I$-regular hyperbolic elements\\ on affine buildings}

\author{Corina Ciobotaru\thanks{cociobotaru@math.au.dk}}
\date{July 15, 2024}

\begin{document}
\maketitle
\begin{abstract}
The first goal of this article is to investigate a refinement of previously-introduced strongly regular hyperbolic automorphisms of locally finite thick Euclidean buildings $\Delta$ of finite Coxeter system $(W,S)$. The new ones are defined for each proper subset $I \subsetneq S$ and called strongly $I$-regular hyperbolic automorphisms of $\Delta$. Generalizing previous results, we show that such elements exist in any group $G$ acting cocompactly and by automorphisms on $\Delta$. Although the dynamics of strongly $I$-regular hyperbolic elements $\gamma$ on the spherical building $\partial_\infty \Delta$ of $\Delta$ is much more complicated than for the strongly regular ones, the $\lim\limits_{n\to \infty} \gamma^{n}(\xi)$ still exists in $\partial_\infty \Delta$ for ideal points $\xi \in \partial_\infty \Delta$ that satisfy certain assumptions. An important role in this business is played by the cone topology on $\Delta \cup \partial_\infty \Delta$ and the projection of specific residues of $\partial_\infty \Delta$ on the ideal boundary of $\Min(\gamma)$.

All the above research is performed in order to achieve the second, and main, goal of the article. Namely, we prove that for closed groups $G$ with a type-preserving and strongly transitive action by automorphisms on $\Delta$, the Chabauty limits of certain closed subgroups of $G$ contain as a normal subgroup the entire unipotent radical of concrete parabolic subgroups of $G$.
\end{abstract}

\tableofcontents

\section{Introduction}

Strongly regular hyperbolic automorphisms of locally finite thick Euclidean buildings $\Delta$ were introduced and studied in \cite{CaCi}. They have a captivating geometry when acting on $\Delta$ and intriguing dynamical properties when acting on the ideal boundary $\partial_\infty \Delta$ of $\Delta$. Still, their key application in \cite{CaCi} is the equivalence between the purely geometric property of a closed subgroup $G \leq \Aut(\Delta)$ of type-preserving automorphisms $G$ to act strongly transitively on $\Delta$, and the harmonic analytic property of the pair $(G,K)$ to be Gelfand, when $K$ is the stabilizer in $G$ of a special vertex of $\Delta$. 

Being a valuable tool in many situations, dynamics demonstrates that it is also applicable to the context of connected reductive groups $G$ defined over non-Archimedean local fields. Namely, we want to employ dynamics when computing the Chabauty limits of the fixed point group $H$ of an involutorial automorphism of $G$. 

First in the series, this paper smoothes the path for the above-mentioned goal. In a nutshell, we define the so-called strongly $I$-regular hyperbolic automorphisms of $\Delta$, study their existence and dynamical properties in a broad setting, and, as a main application, provide very general conditions for Chabauty limits of certain closed subgroups of well-behaved $G \leq \Aut(\Delta)$ to contain as a normal subgroup the entire unipotent radical of concrete parabolic subgroups of $G$.

In order to present our main results, we first need to recall some standard terminology. We let $\Delta$ be a locally finite thick Euclidean building of dimension $\geq 1$. With respect to that, the ideal boundary of $\Delta$ is denoted by $\partial_{\infty} \Delta$. The latter is also endowed with a structure of a spherical building and has associated a finite Coxeter system denoted by $(W,S)$. The group $(W,S)$ is the finite Weyl group associated with $\Delta$. 

Given that data, to any ideal simplex $\sigma$ in $\partial_\infty \Delta$ one can assign in a unique way a subset $I_{\sigma} \subsetneq S$ (including the empty set) called the type of $\sigma$ (see Definition \ref{def::type_simplex}). Each element from the set $I_\sigma$ corresponds to a specific wall of the spherical building $\partial_{\infty} \Delta$, and also to a specific family of parallel walls in the affine building $\Delta$. 

For a given type $I \subsetneq S$ (including the empty set), we consider the set $I(\partial_\infty \Delta)$ of all ideal simplices in $\partial_\infty \Delta$ having type $I$. As explained in Section \ref{subsec::cone_top_type_I}, the set $I(\partial_\infty \Delta)$ is endowed with the cone topology that is compatible with the continuous  action of $\Aut(\Delta)$ on $\Delta \cup \partial_\infty \Delta$. Now, a strongly $I$-regular automorphism of $\Delta$ is a hyperbolic element of $\Aut(\Delta)$ whose bi-infinite geodesic translation axes have their two endpoints lying in the interior of (same) two opposite ideal simplices of $\partial_\infty \Delta$ of type $I$. For the precise definitions and basic properties of those automorphisms the reader can jump to Section \ref{subsec::I_regular_hyp_elem}. When $I=\{\emptyset\}$, one recovers the notion of strongly regular hyperbolic automorphisms of $\Delta$ from \cite{CaCi}.

Section \ref{subsec::existence_strongly_regular} proves the first main result of this article, namely Theorem \ref{thm:ExistenceStronglyReg_1} stated below. It claims the existence of strongly $I$-regular hyperbolic automorphisms in well-behaved closed subgroups $G \leq \Aut(\Delta)$. This is a fair generalization of \cite[Theorem 1.2]{CaCi}.

\begin{theorem}
\label{thm:ExistenceStronglyReg_1}
Let $G$ be a group acting cocompactly by automorphisms on a locally finite thick Euclidean building $X$. Let $(W,S)$ be the associated finite Coxeter system of $X$, and $I \subsetneq S$. Then $G$ contains a strongly $I$-regular hyperbolic element. 
\end{theorem}

\medskip
The intriguing dynamics of strongly $I$-regular hyperbolic automorphisms on $\partial_\infty \Delta$ makes its appearance in Section \ref{subsect:dynamics-reg}. 
After some investigation of projections (defined on $\partial_\infty \Delta$) of residues of type $J \subsetneq S$ on residues of type $I  \subsetneq S$ in $\partial_\infty \Delta$, the outcome of Section \ref{subsect:dynamics-reg} is Proposition \ref{prop::dynamics_I_regular}. That is presented here in a very light form for the sake of the reader and to ease the exposition. 

\begin{proposition}(See Proposition \ref{prop::dynamics_I_regular} for the precise statement.)
\label{prop::dynamics_I_regular_1}
Let $\Delta$ be a Euclidean building, $(W,S)$ the associated finite Coxeter system, $I \subsetneq S$, and $\gamma \in \Aut(\Delta)$ a type-preserving strongly $I$-regular hyperbolic element. 

Then, under some concrete assumptions on $\gamma$ and ideal point $\xi \in \partial_\infty\Delta$, the limit $\lim\limits_{n\to \infty} \gamma^{n}(\xi)$ exists in $\partial_\infty\Delta$ and coincides with a precise point in the ideal boundary $\partial_\infty \Min(\gamma)$ of the set $\Min(\gamma) \subset \Delta$ of $\gamma$-translated points.
\end{proposition}

As announced, strongly $I$-regular hyperbolic automorphisms of $\Delta$ are put to work in Section \ref{sec::applications}. One of the intermediate results, before reaching the main goal of this article Theorem \ref{thm::main_goal}, is Proposition \ref{prop::transit_sigma}. Again, to ease the presentation in the introduction we state both of them in lay terms.

\begin{proposition}(See Proposition \ref{prop::transit_sigma}  for the precise statement.)
\label{prop::transit_sigma_1}
Let $\Delta$ be a locally finite Euclidean building, $(W,S)$ the associated finite Coxeter system, and $I \subsetneq S$. Let $\sigma_+,\sigma_-$ be two opposite ideal simplices in $\partial_\infty \Delta$, with $\sigma_+$ of type $I$. Let $\{\gamma_n\}_{n\in \NN}$ be a sequence of strongly $I$-regular hyperbolic elements of $\Delta$ such that $\sigma_+,\sigma_-$ are the attracting and repelling minimal ideal simplices of $\gamma_n$, for every $n\in \NN$. 

Then, under some concrete assumptions, the sequence $\{\gamma_n\}_{n\in \NN}$ limits any open neighborhood $V \subset I_{\sigma_-}(\partial_\infty \Delta)$ of $\sigma_-$ to the set of all ideal simplices opposite $\sigma_+$, i.e. $\lim\limits_{n \to \infty} \gamma_{n}(V)=\Opp(\sigma_+)$.
 \end{proposition}

Even if at a first glance, the main Theorem \ref{thm::main_goal} appears to hold true under very mysterious assumptions as displayed in Section \ref{subsec::main_goal}, those will become mellow when working with connected reductive groups defined over non-Archimedean local fields. This case will be treated in a forthcoming paper, the next in this series. Ignoring all the wild assumptions, the main result in an idealistic world looks as follows.

\begin{theorem}(See Theorem \ref{thm::main_goal} for the precise statement.)
\label{thm::main_goal_1}
Let $\Delta$ be a locally finite thick Euclidean building with associated finite Coxeter system $(W,S)$, $G$ a closed group of type-preserving automorphisms of $\Delta$ with a strongly transitive action on $\Delta$, and $H$ a closed subgroup of $G$ enjoying further properties. Let $I \subsetneq S$.

Then for any sequence $\{a_n\}_{n \in \NN}$ of strongly $I$-regular hyperbolic elements of $\Delta$ satisfying further properties,  a corresponding Chabauty limit $L$ of $H$ with respect to  $\{a_n\}_{n \in \NN}$ admits the semi-direct product decomposition 
$$ L = U_{I} \rtimes L^{I},$$
where $U_{I}$ is normal in $L$  and the unipotent radical of a specific parabolic subgroup $P_I$ of $G$, and $L^{I} =L \cap M^I$ with $M^I$ the Levi factor of $P_I$.
\end{theorem}

The proof of Theorem \ref{thm::main_goal} uses Proposition \ref{prop::transit_sigma} as well as some of the techniques from \cite{BaWil}.

\medskip
The article is structured in three main sections. The needed background material on groups acting on affine and spherical buildings is given in Section \ref{sec::bg_material}. Beside the general facts related to strongly transitive actions on buildings that are recalled there, Section \ref{sec::bg_material} contains a fair exposition and comparison between the different ways of defining parabolic, contraction and unipotent subgroups. A very short reminder about Chabauty topology is given as well. Section \ref{sec::cone_top} is divided in five subsections. The first two, paving the way towards the wanted dynamics, study stars of ideal simplices, and the cone topology on the set $I(\partial_\infty \Delta)$ of all ideal simplices in $\partial_\infty \Delta$ of a given type $I$. Strongly $I$-regular hyperbolic automorphisms of $\Delta$ are defined and studied in the following two subsections. Their wild dynamical properties are presented in the fifth.  As its name suggests, Section \ref{sec::applications} employes strongly $I$-regular hyperbolic elementes and delivers the desired main goal of the paper.

\label{subsec::apartments_stars}

\subsection*{Acknowledgements}  The author was supported by a research grant (VIL53023) from VILLUM FONDEN.

\section{Background material on groups acting on buildings}
\label{sec::bg_material}

In this article we fix  $\Delta$ to be a locally finite thick Euclidean building of dimension $\geq 1$, thus irreducible. We consider $\Delta$, and all other buildings in this paper, to be endowed with their complete system of apartments. With respect to that, the ideal boundary (also called the visual boundary) of $\Delta$ is denoted by $\partial_{\infty} \Delta$, and this is endowed with a structure of a spherical building. Points of $\Delta$, respectively $\partial_{\infty} \Delta$, will be considered with respect to a geometric realisation of that affine, respectively spherical, building. By \cite[Theorem 11.79]{AB}  the apartments of $\partial_{\infty} \Delta$ are in one to one correspondence with those of $\Delta$.  We denote by $\Aut(\Delta)$ the full group of automorphisms of $\Delta$, and by $\Aut_{0}(\Delta) \leq \Aut(\Delta)$ the subgroup of all \textbf{type-preserving} automorphisms, meaning those automorphisms that preserve a chosen coloring of the vertices of the model chamber and thus of the vertices of the chambers of $\Delta$. 

By a simplex $\sigma$ of $\Delta$, or of $\partial_\infty \Delta$, we mean the interior of it, thus $\sigma$ is considered without its faces, so $\sigma$ is open.  Let $\Ch(\partial_\infty \Delta)$, respectively $\Ch(\partial_\infty  \mathcal{A})$, be the set of all chambers of the spherical building $\partial_\infty \Delta$, respectively spherical apartment $\partial_\infty  \mathcal{A}$. 

For a subgroup $G$ of $\Aut(\Delta)$, and a subset $Y $ of $\Delta$, or of $\partial_\infty \Delta$, we denote
$$\Fix_{G}(Y):=\{ g \in G \; \vert \; g(z)=z, \forall z \in Y \} \text{\; and \; } \Stab_{G}(Y):=\{ g \in G \; \vert \; g(Y)=Y \text{ setwise}\}.$$
 
 Let us fix for what follows a special vertex $x_{0} \in \Delta$ (see~\cite[Sec.~16.1]{Gar97}) and an apartment $\mathcal{A} \subset \Delta$ such that $x_0 \in \mathcal{A}$. In the apartment $\mathcal{A}$ we fix a Weyl chamber $Q(x_0, c_+)\subset \mathcal{A}$, with base point $x_0$. This $Q(x_0, c_+)$ is a sector in $\mathcal{A}$ with base point $x_0$ and having at its infinity the ideal chamber in $\partial_\infty \mathcal{A}$ that will be denoted by $c_+$. The ideal chamber opposite $c_+$ in the apartment $\partial_\infty \mathcal{A}$ will be denoted by $c_-$. The two opposite ideal chambers $c_+,c_-$ of $\partial_\infty \mathcal{A}$ are also fixed for what follows.

\begin{definition}[Stars, roots and boundary walls in buildings]
\label{def::star_roots}
Let $X$ be an affine or spherical building. For a simplex $\sigma \subset X$  let $\st(\sigma, X)$ be the set of all (open) simplices in $X$ that contain  $\sigma$ as a sub-simplex. $\st(\sigma, X)$ is called the \textbf{star of $\sigma$ in $X$}. 

A \textbf{root} $\alpha$ of $X$ is a half-apartment of $X$ and its \textbf{boundary wall} (not the ideal boundary) is denoted by $\partial \alpha$.  Let $\mathcal{A}_{X}(\alpha)$ be the set of all apartments of $X$ that contain $\alpha$ as a half-apartment.  
\end{definition}

\begin{definition}[Positive, negative, simple roots]
\label{def::pos_neg_simple_roots}
For the fixed apartment $\mathcal{A}$, denote by $\Phi$ the set of all roots of $\partial_\infty  \mathcal{A}$, and by $\Phi^{+}$ the roots in $\Phi$ that contain $c_+$ as a chamber; we call $\Phi^{+}$ the set of \textbf{positive roots} with respect to $c_{+}$. Then taking $\Phi^{-}:=\Phi \setminus \Phi^{+}$, one easily notices that $\Phi^{-}$ is the set of all roots of $\Phi$ that contain $c_-$ as a chamber. We call $\Phi^{-}$ the set of \textbf{negative roots} with respect to $c_{+}$.

Denote by $S$ the set of boundary walls associated with the roots in $\Phi^+$ that contain one of the \textbf{panels} of $c_+$. A panel is also called \textbf{facet}, and it is a co-dimension-$1$ simplex of $c_+$. The roots in $\Phi^+$ having their boundary wall in $S$ are called the \textbf{simple roots} of $\partial_\infty  \mathcal{A}$ with respect to $c_+$.
\end{definition}

\begin{definition}[Type of an ideal simplex]
\label{def::type_simplex}
Let $\sigma$ be a simplex in the closure of the ideal chamber $c_+$ of the fixed spherical apartment $\partial_\infty \mathcal{A}$. Denote by 
$$I_\sigma:=\{\partial\alpha \in S  \; \vert \; \sigma \text{ subset of the boundary wall } \partial \alpha, \text{ with  } \alpha \in \Phi^+ \}$$
the subset of boundary walls in $S$ that contain $\sigma$ as a subset of each of those walls.  Notice that $I_{c_+}:=\emptyset$.
The \textbf{type} of $\sigma$ is defined to be $I_{\sigma}$. 
\end{definition}

The \textbf{finite Weyl group W} associated with $\Delta$, and also with $\partial_\infty \Delta$, is the group generated by the reflections of $\partial_\infty \mathcal{A}$ through the panels of $c_+$, thus one can say that $W$ is generated by the elements of $S$. Then it is well known that $(W,S)$ is the finite irreducible Coxeter system of $\Delta$, and also of $\partial \Delta$.

The \textbf{finite Weyl group $W_{I_\sigma}$} corresponding to a type $I_\sigma \subset S$ is the subgroup of $W$ that is generated by the reflections of $\partial_\infty \mathcal{A}$ through the boundary walls from $I_\sigma$.

As every chamber of $\partial_\infty \Delta$ is a genuine copy of the fixed ideal chamber $c_+$, the \textbf{type of a simplex $\gamma$} in $\partial_\infty \Delta$ is defined to be the type $I \subset S$ of the corresponding simplex in the model chamber $c_+$. Then for a type $I \subset S$ we take 
 
$$I(\partial_\infty \Delta) :=\{ \gamma \text{ simplex  of } \partial_\infty \Delta \text{ of type } I\}.$$  If $\sigma$ is the ideal chamber $c_+$, then $I_{c_+}(\partial_\infty \Delta)=\Ch(\partial_\infty \Delta).$

Notice, any two ideal simplices $\sigma$ and $\sigma'$ in $\partial_\infty \Delta$ that are opposite might not have the same type with respect to the finite Weyl group $(W,S)$, but they have the same type with respect to the walls of the affine building, thus with respect to the affine Weyl group of $\Delta$. Sometimes we will abuse this notation and meaning, but it will be clear from the context.

\begin{definition}[$I$-sector in an affine building]
Let $X$ be an affine building, and $\sigma$ be an ideal simplex in $\partial_\infty X$ of type $I \subset S$. An \textbf{$I$-sector} in $X$ is the open cone of dimension $\dim(\sigma)+1$, with base point some $x \in X$, and whose boundary at infinity is exactly $\sigma$. An ${\emptyset}$-sector in $X$ corresponds exaclty to a Weyl chamber, also called Weyl sector. 
\end{definition}

The \textbf{affine Weyl group} associated with the affine building $\Delta$ is defined to be 
$$W^{\text{\tiny{aff}}}:=\Stab_{\Aut_{0}(\Delta)}( \mathcal{A}) / \Fix_{\Aut_{0}(\Delta)}( \mathcal{A}).$$
Again, we take $S^{\text{\tiny{aff}}}$ to be the set of the reflections of $ \mathcal{A}$ through the walls of a chamber in $\Ch( \mathcal{A})$ chosen in advanced. Then, it is known that $W^{\text{\tiny{aff}}}$ is generated by $S^{\text{\tiny{aff}}}$, and that $(W^{\text{\tiny{aff}}},S^{\text{\tiny{aff}}})$ is the affine irreducible Coxeter system of $\Delta$. 
Another well-known fact is that $W^{\text{\tiny{aff}}}$ contains a maximal abelian normal subgroup $T$ isomorphic to $\mathbb{Z}^{m}$, whose elements are Euclidean translation automorphisms of the apartment $ \mathcal{A}$, and where $m$ is the Euclidean dimension of the affine apartment $ \mathcal{A}$. The elements in $T$ are induced by hyperbolic automorphisms of $\Delta$. In particular, every element of $T$ can be lifted (not in a unique way) to a hyperbolic element in $\Aut_{0}(\Delta)$. A lift of $T$ in $\Aut_{0}(\Delta)$ is not abelian in general.

\subsection{General facts on strong transitivity}

In this section we consider $G$ to be a closed subgroup of $\Aut_{0}(\Delta)$.

For an ideal simplex $\sigma$ of $\partial_{\infty} \Delta$ we take
$$G_{\sigma}:= \Fix_{G}(\sigma)= \{ g \in G \; \vert \; g(\sigma)=\sigma \text{ pointwise}\}$$ and call it the \textbf{parabolic subgroup of $G$ associated with $\sigma$}. Notice $G_{\sigma}$ is closed in $G$, and since $G$ is type-preserving we have that $G_{\sigma} =\Fix_{G}(\sigma)= \Stab_{G}(\sigma)$. 

The subgroup $G_{\sigma}$ is maximal when $\sigma$ is an ideal vertex in the spherical building $\partial_{\infty} \Delta$, and minimal when $\sigma$ is an ideal chamber in the spherical building $\partial_{\infty} \Delta$. 

When $\sigma=c$ we call $B(G):= G_{c}$ to be a \textbf{Borel subgroup} of $G$. We also take 
$$B^0(G):=\{ g \in G \; \vert \; g(c)=c, \; g \text{ elliptic}\}.$$

A \textbf{good maximal compact} subgroup of $G$ is the stabilizer $\Stab_{G}(x):=\{ g \in G \; \vert \; g(x)=x\}$ of a special vertex $x$ of $\Delta$.   

\begin{definition}[Strong transitivity]
We say that a closed subgroup $G$ of $\Aut_{0}(\Delta)$ acts \textbf{strongly transitively} on $\Delta$ (respectively $\partial_{\infty} \Delta$), if for any two pairs $(\mathcal{A}_1,c_1 )$ and $(\mathcal{A}_2 , c_2)$, consisting each of an apartment $\mathcal{A}_i$ of $\Delta$ (respectively $\partial_{\infty} \Delta$), and a chamber $c_i \in  \Ch(\mathcal{A}_i)$, there exists $g \in G$ such that $g(\mathcal{A}_1) = \mathcal{A}_2$ and $g(c_1) = c_2$.
\end{definition}

It is well known that any semi-simple algebraic group over a non-Archimedean local field acts by type-preserving automorphisms and strongly transitively on its associated Bruhat--Tits building.

When the closed subgroup $G \leq \Aut_{0}(\Delta)$ acts on $\Delta$ strongly transitively, it is a known, easy-to-see, fact that $G$ induces a strongly transitively action on the visual boundary $\partial_\infty \Delta$, see \cite[17.1]{Gar97}. Then in this case a Borel subgroup $B(G)$ of $G$ is unique, up to conjugation in $G$.

For the fixed apartment $ \mathcal{A}$ of $\Delta$, the special vertex $x_0 \in  \mathcal{A}$, and the ideal chamber $c_+ \in \Ch(\partial_\infty  \mathcal{A})$, let $$K(G):=\Stab_G(x_0) \; \text{ and } \; B(G)^{+}:= G_{c_+},$$
and $T(G)$ be a lift in $\Stab_G(\mathcal{A})$ of the maximal abelian normal subgroup $T$ of the affine Weyl group $W^{\text{\tiny{aff}}}$ of $\Delta$. Also, since $T(G)$ acts by translations on the apartment $\mathcal{A}$, one can see that it fixes pointwise the ideal apartment $\partial_\infty \mathcal{A}$, and so $T(G) \leq G_{c_+} \cap G_{c_-}$.  Notice that all those closed subgroup of $G$ depend on the choice of the apartment $ \mathcal{A}$, the special vertex $x_0 \in  \mathcal{A}$ and the ideal chamber $c_+ \in \Ch(\partial_\infty  \mathcal{A})$.  In particular, the decompositions below also depend on the choice of $ \mathcal{A}, x_0$ and $c_+$.  For simplicity, in the next two results we just take $K= K(G), B=B(G)^{+}$, and $T= T(G)$.

\begin{lemma}(\textbf{Cartan decomposition}, for a proof see~\cite[Lemma~4.7]{Cio})
\label{lem::polar_decom}
Let $G$ be closed, strongly transitive and type-preserving subgroup of $\Aut_0(\Delta)$. Then $G=K T K$.
\end{lemma}

\begin{lemma}(\textbf{Iwasawa decomposition}, for a proof see~\cite[p.~14]{Cio})
\label{lem::Iwasawa_decom}
Let $G$ be a closed, strongly transitive and type-preserving subgroup of $\Aut(\Delta)$,  that is not necessarily closed. Then $G=K B$ and $B=TB^0$.
\end{lemma}



\subsection{Parabolics, contraction subgroups, and unipotent radicals}
\label{subsec::para_cont_unit}
We recall now the definitions for the geometric counterpart of parabolic, and the contraction subgroups from \cite{BaWil} when $G$ is a closed, type-preserving subgroup of $\Aut(\Delta)$, and $\gamma \in \Stab_G(\mathcal{A})$ is a hyperbolic automorphism of $\Delta$. One can define:

\begin{equation} 
\label{equ::parab_subgroup}
P^{+}_{\gamma}:=\{ g \in G \; \vert \; \{\gamma^{-n}g \gamma^{n}\}_{n \in \mathbb{N}} \text{ is bounded}\}. 
\end{equation} Following \cite[Section~3]{BaWil}, $P^{+}_{\gamma}$ is a closed subgroup of $G$. One considers also the subgroup
\begin{equation}
\label{equ::contr_subgroup}
U^{+}_{\gamma}:=\{ g \in G \; \vert \; \lim_{n \to \infty }\gamma^{-n}g \gamma^{n}=e\}. 
\end{equation}
In the same way, but using $\gamma^{n}g \gamma^{-n}$, we define $P^{-}_{\gamma}$, and $U^{-}_{\gamma}$.
By \cite[Section~3]{BaWil}, $P^{+}_{\gamma}$ and $U^{+}_{\gamma}$ are called the \textbf{positive parabolic}, respectively the \textbf{positive contraction}, subgroups associated with $\gamma$. In general $U^{+}_{\gamma}$ is not closed. Moreover, from the algebraic point of view, we have the following \textbf{Levi decomposition}, stated here in the particular case of locally compact groups acting on affine buildings:

\begin{theorem}(See \cite[Proposition~3.4, Corollary~3.17]{BaWil})
\label{thm::levi_decom}
Let $G$ be a closed, type-preserving subgroup of $\Aut(\Delta)$, and $\gamma \in \Stab_G(\mathcal{A})$ a hyperbolic automorphism of $\Delta$.
Then $P^{+}_{\gamma}= U^{+}_{\gamma} M_\gamma$, where $M_{\gamma}:= P^{+}_{\gamma} \cap P^{-}_{\gamma}$. Moreover, $U^{+}_{\gamma}$ is normal in $P^{+}_{\gamma}$.
\end{theorem}

In general, the product $P^{+}_{\gamma}= U^{+}_{\gamma} M_\gamma$ might not be semidirect or direct.
We also have the following results from \cite{Cio}. 

\begin{proposition}(See \cite[Proposition 4.15]{Cio})
\label{prop::geom_levi_decom}
Let $G$ be a closed, non-compact and type-preserving subgroup of $\Aut(\Delta)$. Let $\mathcal{A}$ be an apartment in $\Delta$ and assume there exists a hyperbolic automorphism $\gamma \in \Stab_{G}(\mathcal{A})$. Denote the attracting and repelling endpoints of $\gamma$ by $\xi_+, \xi_- \in \partial_\infty \mathcal{A}$, and let $\sigma_+$, respectively $\sigma_-$, be the unique simplex in $\partial_\infty \mathcal{A}$ such that $\xi_+$, respectively $\xi_-$, is contained in the interior of $\sigma_+$, respectively $\sigma_-$. 

Then $P^{+}_{\gamma} = G_{\sigma_+}= G_{\xi_+}$. In particular, we obtain that $P^{+}_{\gamma} \cap P^{-}_{\gamma} = \Fix_G(\{\xi_-, \xi_+\})= \Fix_G(\{\sigma_-, \sigma_+\})$. 
\end{proposition}

The geometric Levi decomposition is as follows. 
\begin{corollary}(See \cite[Corollary 4.17]{Cio})
\label{coro::geom_levi_decom}
Let $G$ be a closed and type-preserving subgroup of $\Aut(\Delta)$. Let $\mathcal{A}$ be an apartment in $\Delta$ and assume there exists a hyperbolic automorphism $\gamma \in \Stab_{G}(\mathcal{A})$.  Denote the attracting endpoint of $\gamma$ by $\xi_+ \in \partial_\infty \mathcal{A}$. Let $\sigma$ be a simplex in $\partial_\infty \mathcal{A}$ such that $\xi_+$ is contained in the closure of $\sigma$ (thus $\sigma$ may not be the minimal one).  Let also $c \in \Ch(\partial_\infty \mathcal{A})$ with $\xi_+$ in the closure of $c$.

Then $G_{\sigma}= U^{+}_{\gamma} (M_\gamma \cap G_{\sigma})$, where $M_{\gamma}:= P^{+}_{\gamma} \cap P^{-}_{\gamma}$. In particular,  $G_{c}^{0}= U^{+}_{\gamma} (M_\gamma \cap G_{c}^{0})$. In addition, $U^{+}_{\gamma} $ is normal in $G_{c}^{0}$, respectively $G_{\sigma}$, with $U^{+}_{\gamma} \cap M_\gamma$ compact subgroup.
\end{corollary}

\medskip
We define now the analog to the algebraic unipotent radical of a parabolic subgroup of an algebraic group. 
\begin{definition}[Root groups]
\label{def::H_plus}
 Let $G$ be a closed subgroup of $\Aut_{0}(\Delta)$ and $\alpha$ any root of $\partial_{\infty} \Delta$. The \textbf{root group}  $U_{\alpha}(G)$ associated with $\alpha$ and $G$ is defined to be the set of all elements $g$ of $G$ which fix $\st_{\partial_{\infty} \Delta}(P)$ pointwise for every panel $P$ in $\alpha - \partial \alpha$. For $G=\Aut_{0}(\Delta)$ we use the notation $U_{\alpha}:=U_{\alpha}(\Aut_{0}(\Delta))$.

\end{definition} 

\begin{remark}
Let $G$ be a closed subgroup of $\Aut_{0}(\Delta)$ and $\alpha$ a root of $\partial_{\infty} \Delta$. Then, as the action of $G$ on $\partial_{\infty} \Delta$ is continuous, we have $U_{\alpha}(G)$ is a closed subgroup of $G$. 
\end{remark}

\begin{remark}
Since $\Delta$ is an irreducible affine building, one can easily verify that for every proper subset $I$ of $S$, including the empty set $\emptyset$, there exists a unique simplex $\sigma$ of $c_+$ whose type is exaclty $I$.
\end{remark}

For any proper subset $I$ of $S$, including the empty set $\emptyset$, consider the corresponding simplex $\sigma$ of $c_+$ of type $I$, and let $\mathcal{A}_I := \mathcal{A} \cap \St(\sigma, \partial_\infty \Delta)$ (an apartment of $\St(\sigma, \partial_\infty \Delta)$). Then we define
$$U^{+}_{I}(G):= \langle U_{\alpha}(G) \; \vert \; \mathcal{A}_I \subset \alpha  \in \Phi^{+} \rangle \text{ and } U^{-}_{I}(G):= \langle U_{\alpha}(G) \; \vert \;  \mathcal{A}_I \subset - \alpha  \in \Phi^{+}\rangle$$
and call them the \textbf{standard positive}, respectively \textbf{standard negative}, \textbf{unipotent radical of type $I$} of $G$. 

For $I = \{\emptyset\}$ we simply use the notation $U^{\pm}(G):= U^{\pm}_{\emptyset}(G)$. When $G=\Aut_{0}(\Delta)$, we simply use the notation $U^{\pm}_{I}: = U^{\pm}_{I}(\Aut_{0}(\Delta))$. Notice that $U^{\pm}_{I}(G) \leq U^{\pm}(G)$, and $U^{\pm}_{I}(G) \leq U^{\pm}_{I}$, for every $I$ proper subset of $S$.

\begin{remark}
By an easy verification one can show that $U^{+}_{I}(G) \leq G_{\sigma}$.
\end{remark}

\begin{definition}[Moufang groups]
\label{def::moufang}
Let $H$ be a closed subgroup of $\Aut_{0}(\Delta)$. We say $H$ is  \textbf{Moufang} if for every root $\alpha$ of $\partial_{\infty} \Delta$ the associated root group $U_{\alpha}(H)$ acts transitively on $\mathcal{A}_{\partial_{\infty} \Delta}(\alpha)$.  
\end{definition}

Then we know that:
\begin{theorem}(See \cite[Theorem 1.5]{Cio_M})
\label{main_thm2}
Let $\Delta$ be a locally finite thick affine building of dimension $\geq 2$. Let $H$ be a closed subgroup of $\Aut_{0}(\Delta)$ that acts strongly transitively on $\Delta$. Then $H$ is Moufang.  
\end{theorem}

\begin{theorem}(See \cite[Lemma 7.25(3)]{AB})
\label{cor::simply_tran}
Let $\Delta$ be an irreducible, locally finite thick affine building of dimension $\geq 2$. Let $H$ be a closed subgroup of $\Aut_{0}(\Delta)$ and  $\alpha$ any root of $\partial_{\infty} \Delta$. Suppose the root group $U_{\alpha}(H)$  acts transitively on $\mathcal{A}_{\partial_{\infty} \Delta}(\alpha)$. Then the action of $U_{\alpha}(H)$ on $\mathcal{A}_{\partial_{\infty} \Delta}(\alpha)$ is simple-transitive; i.e., for every $u \in U_{\alpha}(H) - \{\id\}$ the action of $u$ has no fixed points on $\mathcal{A}_{\partial_{\infty} \Delta}(\alpha)$.
\end{theorem}

 By \cite[Chapter 6, Thms. 6.17 and 6.18]{Ro}, \cite[Theorem 1.5, Corollary 2.3]{Cio_M} one gets the following Levi decomposition type. 

\begin{proposition}
\label{prop::normal_U}
Let $\Delta$ be an irreducible, locally finite thick affine building of dimension $\geq 2$ such that $\Aut_{0}(\Delta)$ is strongly transitive on $\Delta$.
Let  $G$ be a closed and strongly transitive subgroup of $\Aut_{0}(\Delta)$. Let $\sigma_+$ be a simplex of $c_+$ of type $I$, with opposite $\sigma_-$ in $\partial_\infty \mathcal{A}$. 

Then $U^{\pm}_I = U^{\pm}_I(G)$ and  $G_{\sigma_\pm} =  U^{\pm}_I \rtimes M(G)$, where $M(G):= G_{\sigma_+} \cap G_{\sigma_-}$.  In particular, $U^{\pm}_I$ is a normal subgroup of $G_{\sigma_\pm}$, $U^{\pm}_I \cap M(G) =\Id$, and $U^{\pm}_I$ acts simple-transitively on $\Opp(\sigma_\pm)$, the set of all simplices of $\partial_\infty \Delta$ opposite $\sigma_\pm$.
\end{proposition}
\begin{proof}
By Theorem \ref{main_thm2} applied to the group $\Aut_{0}(\Delta)$ we get that $\Aut_{0}(\Delta)$ is Moufang. So by Theorem \ref{cor::simply_tran}, every root group $U_{\alpha}$ acts simple-transitively on $\mathcal{A}_{\partial_{\infty} \Delta}(\alpha)$. Since $G$ is also Moufang, again by Theorem \ref{main_thm2} applied to $G$, we get that $U_{\alpha}= U_{\alpha}(G)$, in particular $U^{\pm}_I = U^{\pm}_I(G)$.

Now by \cite[Chapter 6, Thms. 6.17 and 6.18]{Ro}, the conclusion of the proposition is true for $G= \Aut_{0}(\Delta)$. Since $G_{\sigma_\pm} \leq  \Aut_{0}(\Delta)_{\sigma_\pm}$, and $U^{\pm}_I$ normal in $\Aut_{0}(\Delta)_{\sigma_\pm}$, thus in $G_{\sigma_\pm}$ as well, and $U^{\pm}_I$ acts simple-transitively on $\Opp(\sigma_\pm)$, we obtain the conclusion for $G$ as well.

For the simple-transitively of $U^{\pm}_I$ on $\Opp(\sigma_\pm)$ one can also look at \cite[Corollary 7.67]{AB}.
\end{proof}

\begin{remark}
One can ask what the relation is between the unipotent radicals $U_I^{\pm}$ of type $I$ and the contraction subgroup $U_\gamma^{\pm}$, given that we have the equality $P^{\pm}_{\gamma} = G_{\sigma_\pm}$ and also the very similar associated Levi decompositions. For general closed groups $G \leq \Aut_{0}(\Delta)$, even for those that are Moufang, it is not clear when $U_I^{\pm}$ is a subgroup of $U_\gamma^{\pm}$, and in principal, $U_\gamma^{\pm}$ can be larger than $U_I^{\pm}$. But when $G$ is the group of $k$-rational points of a simple algebraic group over a local
field $k$, the contraction groups equal the unipotent radical, i.e. $U_\gamma^{\pm}=U_I^{\pm}$ ( see \cite[Lemma 2.4]{Pra}, or \cite[Prop. 7.4.33]{BrTi_72}, or \cite[Corollaries 4.28, 4.29]{Cio}).
\end{remark}

\subsection{Short on Chabauty topology}
For a locally compact topological space $X$, the set $\mathcal{F}(X)$ of all closed subsets of $X$ is a compact topological space with respect to the Chabauty topology (\cite[Proposition~1.7, p.~58]{CoPau}). Given a family $\mathcal{T} \subset \mathcal{F}(X)$ it is natural to ask what is the closure  $\overline{\mathcal{T}}^{Ch}$ of $\mathcal{T}$ with respect to the Chabauty topology, and whether or not elements of $\overline {\mathcal{T}}^{Ch} \setminus \mathcal{T} $ satisfy the same properties as the ones of  $\mathcal{T}$.  We call elements of $\overline{\mathcal{T}}^{Ch}$ the \textbf{Chabauty limits of $\mathcal{T}$}, and  $\overline{\mathcal{T}}^{Ch}$ the  \textbf{Chabauty compactification} of $\mathcal{T}$ (see \cite{Ch}). 

For a locally compact group $G$ we denote by $\mathcal{S}(G)$ the set of all closed subgroups of $G$. By \cite[Proposition~1.7, p.~58]{CoPau} the space $\mathcal{S}(G)$ is closed in $\mathcal{F}(G)$, with respect to the Chabauty topology, and thus compact.  

\begin{proposition}(\cite[Proposition~1.8, p.~60]{CoPau}, \cite[Proposition I.3.1.3]{CEM})
\label{prop::chabauty_conv}
 Suppose $X$ is a locally compact metric space.
A sequence of closed subsets $\{F_n\}_{n \in \NN} \subset \mathcal{F}(X)$  converges to $F \in \mathcal{F}(X)$ if and only if the following two conditions are satisfied:
\begin{itemize} 
\item[1)] For every $f \in F$ there is a sequence $\{f_n \in F_n\}_{n \in \NN}$ converging to $f$;
\item[2)] For every sequence $\{f_n \in F_n\}_{n \in \NN}$, if there is a strictly increasing subsequence $\{n_k\}_{k \in \NN}$ such that $\{f_{n_k} \in F_{n_k}\}_{k \in \NN}$ converges to $f$, then $f \in F$.
\end{itemize}
\end{proposition}

Proposition~\ref{prop::chabauty_conv} can be applied to a sequence of closed subgroups $\{H_n\}_{n \in \NN} \subset \mathcal{S}(G)$ converging to $H \in \mathcal{S}(G)$, obtaining a similar characterisation of convergence in $\mathcal{S}(G)$, when $G$ is a locally compact group that is endowed with a metric inducing its topology.


\section{On the cone topology on $\partial_\infty \Delta$ and dynamical properties}
\label{sec::cone_top}

\subsection{Apartments and stars of ideal simplices}
\label{subsec::apartments_stars}

In this subsection we define and give some properties of the `stars' and the opposites of ideal simplices in the ideal boundary $\partial_\infty\Delta$ of a locally finite thick Euclidean building $\Delta$. We show that an apartment of $\Delta$ determines uniquely, respectively, is uniquely determined by, the stars of two opposite ideal simplices.

Let $\alpha:= \{a_n\}_{n \geq 1}$ be a sequence in $G$ that escapes every compact subset of $G$. Then, as the space $\Delta \cup \partial_{\infty} \Delta$ is compact with respect to the cone topology, the sequence $\{a_n\}_{n \geq 1}$ admits a subsequence $\beta:=\{a_{n_k}\}_{k\geq 1}$ with the property that $\{a_{n_k}(x_0)\}_{k\geq 1}$ converges to an ideal point $\xi \in \partial_{\infty} \Delta$. Notice, there is a unique minimal ideal simplex $\sigma \subset \partial_{\infty} \Delta$ such that $\xi$ is in the interior of $\sigma$. One can have anything in-between  $\sigma$ being an ideal chamber or $\sigma$ being an ideal vertex of the spherical building $\partial_{\infty} \Delta$.

\begin{definition}[Stars in apartments and $\sigma$-cones]
\label{def::st_cone_N}
Let $\sigma$ be an ideal simplex in $\partial_{\infty} \Delta$ and choose an apartment $\mathcal{B}$ of $\Delta$  such that $\sigma \subset \partial_{\infty} \mathcal{B}$. Let $x \in \mathcal{B}$ be a point. We set 
$$\st(\sigma,\partial_{\infty} \mathcal{B}):= \{c \in \Ch(\partial_{\infty} \mathcal{B}) \; \vert \; \sigma \text{ in the closure of } c\}, \text{ resp.,}$$
$$ \st(x, \mathcal{B}):= \{c \in \Ch(\mathcal{B}) \; \vert \; x \text{ in the closure of } c\},$$
 and call them the \textbf{$\partial_{\infty} \mathcal{B}$-star of $\sigma$}, respectively, the \textbf{$\mathcal{B}$-star of $x$}.  As well,  the associated \textbf{$\sigma$-cone in $\mathcal{B}$ with base-point $x \in \mathcal{B}$} is defined by $Q(x, \sigma, \mathcal{B}):= \bigcup\limits_{c \in \st(\sigma,\partial_{\infty} \mathcal{B})} Q(x, c)$, where $Q(x, c) \subset \mathcal{B}$ is the open Weyl sector emanating from the point $x$ and having $c$ as ideal chamber.  A  \textbf{$\sigma$-cone in $\mathcal{B}$} is  $Q(x, \sigma,\mathcal{B})$ for some $x \in \mathcal{B}$.
\end{definition}

\begin{remark}
\begin{enumerate}
\item
 Notice, $\st(\sigma,\partial_{\infty} \mathcal{B})$, respectively, $\st(x, \mathcal{B})$, are both a finite set of ideal chambers, respectively, chambers.
\item
 As well, the intersection of two $\sigma$-cones in $\mathcal{B}$ contains a $\sigma$-cone in $\mathcal{B}$.
\item
 Moreover, if we take $\sigma_-$ to be the unique simplex in $\partial \mathcal{B}$ opposite $\sigma$, then $\st(\sigma,\partial_{\infty} \mathcal{B})$ is opposite $\st(\sigma_-,\partial_{\infty} \mathcal{B})$ in the sense that any ideal chamber $c \in \st(\sigma,\partial_{\infty} \mathcal{B})$ has a unique opposite in  $\st(\sigma_-,\partial_{\infty} \mathcal{B})$, and vice-versa.
\end{enumerate}
\end{remark}

The following remark recalls some facts related to residues and projections (see \cite[Section 5.3]{AB}).
\begin{remark}
\begin{enumerate}
\item
Given an ideal simplex $\sigma_+$ in $\partial_\infty \Delta$ of type $I$, the star $\st(\sigma_+, \partial_\infty \Delta)$ of $\sigma$ in $\partial_\infty \Delta$ is a residue of type $I$ in $\partial_\infty \Delta$. Also, if $\sigma_-$ is an ideal simplex opposite $\sigma_+$, then  $\st(\sigma_-, \partial_\infty \Delta)$ is again a residue of some type $J \subset S$ in $\partial_\infty \Delta$. 
\item
Given an ideal chamber $c$ and a residue $\mathcal{R}$ in $\partial_\infty \Delta$, the projection of $c$ onto $\mathcal{R}$, denoted by $\proj_{\mathcal{R}}(c)$, is the unique chamber in $\mathcal{R}$ at minimal distance from $c$. 
\end{enumerate}
\end{remark}


For the following lemma we will consider the star of a simplex as a set of chambers.

\begin{lemma}
\label{lem::lem_36}
Let $\Delta$ be a locally finite thick Euclidean building, $\sigma_{-}$ and $\sigma_{+}$ two opposite ideal simplices of $\partial_\infty \Delta$, and $\mathcal{B}$ an apartment of $\Delta$ such that $\sigma_{-}, \sigma_{+} \subset \partial_\infty \mathcal{B}$. Then among the $\partial_{\infty} \mathcal{B}'$-stars of $\sigma_-$, i.e. $\st(\sigma_-,\partial_{\infty} \mathcal{B}') \subset \st(\sigma_-, \partial_\infty \Delta)$ with $\mathcal{B}'$ an apartment of $\Delta$,  $\st(\sigma_-,\partial_{\infty} \mathcal{B})$ is the unique one with this property that $\st(\sigma_+,\partial_{\infty} \mathcal{B})$ and $\st(\sigma_-,\partial_{\infty} \mathcal{B'})$ can be part of the same ideal apartment in $\partial_\infty \Delta$. In particular, $\mathcal{B}$ is uniquely determined by, respectively, determines uniquely, $\st(\sigma_+,\partial_{\infty} \mathcal{B})$ and $\st(\sigma_-,\partial_{\infty} \mathcal{B})$.
\end{lemma}

\begin{proof}
Let $c \in \st(\sigma_+,\partial_{\infty} \mathcal{B})$ and $\mathcal{B}'$ be an apartment of $\Delta$ such that $c$ and $\sigma_{-}$ both belong to $\partial_\infty \mathcal{B}'$. Then there is a unique ideal chamber $c'_-$ of $\partial_\infty \mathcal{B}'$ which contains $\sigma_{-}$ in its closure, and it is opposite $c$. It is a well know fact that any two opposite ideal chambers uniquely determine, respectively, are uniquely determined by, an apartment of $\Delta$. 

Let $\mathcal{B}'$ be an apartment of $\Delta$ with $\sigma_{-} \subset \partial_\infty \mathcal{B}'$. Suppose $\st(\sigma_-,\partial_{\infty} \mathcal{B'})$ and $\st(\sigma_+,\partial_{\infty} \mathcal{B})$ are part of the same ideal apartment $\partial_{\infty} \mathcal{B''}$, for some apartment $\mathcal{B}''$ of $\Delta$. We claim that $\mathcal{B''}=\mathcal{B}$. Indeed, let $c \in \st(\sigma_+,\partial_{\infty} \mathcal{B})$. Consider the projection $d:=\proj_{\st(\sigma_-, \partial_\infty \Delta)}(c)$. Then since $c, \sigma_+,\sigma_-$ are all in $\partial_\infty \mathcal{B}$ and $\partial_\infty \mathcal{B}''$, we have that $\proj_{\st(\sigma_-, \partial_\infty \Delta)}(c)$ is in both $\partial_\infty \mathcal{B}$ and $\partial_\infty \mathcal{B}''$; apply  \cite[Lemma 5.36]{AB} and the fact that in each of the apartments $\partial_{\infty} \mathcal{B}$ and $ \partial_{\infty} \mathcal{B''}$  there is a unique chamber at distance $$w_1:= \min\{\delta(\st(\sigma_+,\partial_{\infty} \mathcal{B}), \st(\sigma_-,\partial_{\infty} \mathcal{B}'))\} \in W \text{ from } c.$$
 Moreover, by the same \cite[Lemma 5.36]{AB} we also have that $\proj_{\st(\sigma_+, \partial_\infty \Delta)}(d)=c$. Thus, for every $c \in \st(\sigma_+,\partial_{\infty} \mathcal{B})$ the projection $\proj_{\st(\sigma_-, \partial_\infty \Delta)}(c)$ is in both $\partial_\infty \mathcal{B}$ and $\partial_\infty \mathcal{B}''$. From here one can easily conclude that $\st(\sigma_-,\partial_{\infty} \mathcal{B'})$ is in both $\partial_\infty \mathcal{B}$ and $\partial_\infty \mathcal{B}''$, so $\st(\sigma_-,\partial_{\infty} \mathcal{B'})= \st(\sigma_-,\partial_{\infty} \mathcal{B})$, and $\mathcal{B''}=\mathcal{B}$.
\end{proof}

\subsection{The cone topology on $I_{\sigma}(\partial_\infty \Delta)$}
\label{subsec::cone_top_type_I}

 For each ideal simplex $\sigma$ we choose an ideal point $\xi_\sigma \in \sigma$, called its \textbf{barycenter}. We impose that any type-preserving automorphism $g$ between ideal apartments in $\partial_\infty \Delta$ permutes these barycenters: $g(\xi_\sigma)=\xi_{g(\sigma)}$.
Moreover, for any two opposite ideal simplices $\sigma_+$ and $\sigma_{-}$, we impose that the corresponding barycenters $\xi_{\sigma_+}$ and $\xi_{\sigma_{-}}$ are also opposite. Actually, we just need to choose $\xi_\sigma$ for $\sigma$ in the fundamental ideal chamber of $\partial_\infty \Delta$.  Since we consider that $\Aut_{0}(\Delta)$ acts strongly transitively on $\partial_\infty \Delta$ then $\xi_{\sigma_1}$ will be uniquely and well defined for any simplex $\sigma_1 \in \partial_\infty \Delta$ by the above conditions.

We recall the following definition regarding the cone topology on $\Delta \cup \partial_\infty \Delta$, (see \cite[Part II, Chapter 8]{BH99} or \cite{CMRH} for the general case of masure). For the notation the reader can go back to Section \ref{sec::bg_material}.

\begin{definition}
\label{def::standard_open_neigh_cone_top}
Let  $x \in \Delta$ be a point, $\sigma \subset \bd_\infty \Delta$ be an ideal simplex, and $\sigma' \in I_{\sigma}(\partial_\infty \Delta)$. Let $\xi_{\sigma'}$ be the barycenter of $\sigma'$. It is a well-known fact that there exists an apartment $\mathcal{A} \subset \Delta$ such that $x \in \mathcal{A}$ and $\sigma' \subset \partial_\infty \mathcal{A}$.  Consider the
ray $[x, \xi_{\sigma'}) \subset \mathcal{A}$ issuing from $x$ and corresponding to the ideal point $\xi_{\sigma'}$ of $\sigma'$. Denote by $Q(x,\sigma')$ the cone in $\mathcal{A}$ with base point $x$ and ideal boundary $\sigma'$.  We have that $Q(x,\sigma')$ is an $I_{\sigma}$-sector with base point $x$. Notice $Q(x,\sigma')$ is different than the $\sigma'$-cone $Q(x,\sigma',\mathcal{A})$ in $\mathcal{A}$. Let $r\in [x, \xi_{\sigma'})$; the following subset of $\partial_\infty \Delta$ is defined:
$$
U_{x, r, {\sigma'}}:= \{ \sigma'' \in I_{\sigma}(\partial_\infty \Delta) \; \vert \; [x, r]
\subset [x, \xi_{\sigma'}) \cap [x, \xi_{\sigma''})\}.
$$
The subset $U_{x, r, \sigma'}$ is called a \textbf{standard open neighborhood} in $I_{\sigma}(\partial_\infty \Delta)$ of the (open) simplex $\sigma' \in I_{\sigma}(\partial_\infty \Delta)$ with base point $x$ and gate $r$.
\end{definition}

Notice the cone $Q(x,\sigma')$ defined in Definition \ref{def::standard_open_neigh_cone_top} is a convex set in the apartment $\mathcal{A}$, and $[x, \xi_{\sigma'})$ is a ray in the interior of $Q(x,\sigma')$.

\begin{definition}
\label{def::cone_top_on_chambers}
Let $\Delta$ be a locally finite thick Euclidean building, $x$ a point of $\Delta$, and $\sigma \subset \partial_\infty \Delta$ an ideal simplex. The \textbf{cone topology} $\Top_{x}(I_{\sigma}(\partial_\infty \Delta))$ on $I_{\sigma}(\partial_\infty \Delta)$, with base point $x$, is the topology generated by the standard open neighborhoods $U_{x, r, \sigma'}$ with $\sigma' \in I_{\sigma}(\partial_\infty \Delta)$ and $r\in [x, \xi_{\sigma'})$.
\end{definition}

\begin{lemma}\label{lem:3.3} Let $\Delta$ be a locally finite thick Euclidean building, $x$ a point of $\Delta$, and $\sigma \subset \partial_\infty \Delta$ an ideal simplex. The cone topology $\Top_{x}( I_{\sigma}(\partial_\infty \Delta))$ does not depend on the choice of the barycenters.
\end{lemma}

\begin{proof} Let us consider another choice $\eta$ for the family of barycenters.
Then, considering the cone topology with respect to the family of baricenters $\xi$, if  $\sigma'' \in U_{x, r, \sigma'}^{\xi}$, for some $r$ large enough, then $Q(x,\sigma'')\cap Q(x,\sigma')$ is a convex set, it is a ``large'' portion of $Q(x,\sigma')$, and contains $[x,r]$. 
In particular, $Q(x,\sigma'')\cap Q(x,\sigma')$ contains some large $r'\in [x,\eta_{\sigma'})$.
By the last axiom of a building, there is a Weyl isomorphism of apartments $\psi$ that sends $Q(x,\sigma')$ to $ Q(x,\sigma'')$ and fixes $[x,r]$.
Then $\psi$ sends $\eta_{\sigma'}$ to $\eta_{\sigma''}$ and $[x,\eta_{\sigma'})$ to $[x,\eta_{\sigma''})$.
So $r'\in [x,\eta_{\sigma''})$ and $\sigma'' \in  U_{x, r', \sigma'}^\eta$.
We have proved that the cone topology associated with the family $\xi$ is finer than the one associated with $\eta$. Now, by symmetry, the result follows.
\end{proof}

The following is well known for \cat spaces, thus also for affine buildings (see \cite[Part II, Chapter 8]{BH99}).

\begin{proposition}
\label{prop::independence_base_point}
Let $\Delta$ be a locally finite thick Euclidean building, and $\sigma \subset \partial_\infty \Delta$ an ideal simplex. Let $x,y \in \Delta$ be two different points. Then the cone topologies $\Top_{x}(I_{\sigma}(\partial_\infty \Delta))$ and $\Top_{y}( I_{\sigma}(\partial_\infty \Delta))$ are the same.
 \end{proposition}

Another useful lemma is the following.

\begin{lemma}
\label{lem::open_open_I}
Let $\Delta$ be a locally finite thick Euclidean building, $(W,S)$ the associated irreducible spherical Coxeter system, and $I \subsetneq S$. Let $c$ be an ideal chamber in $\partial_\infty \Delta$, and $\sigma$ a simplex of $c$ of type $I$. If $V_c$ is an open neighborhood of $c$ with respect to the cone topology on $\Ch(\partial_\infty \Delta)$, then $V_\sigma:=\{ \sigma' \subset V_c \; \vert \; \sigma' \text{ of type I} \}$ is an open neighborhood of $\sigma$ with respect to the cone topology on $I(\partial_\infty \Delta)$.
\end{lemma}
\begin{proof}
By the definition of the cone topology, it is enough to consider $V_c$ as a standard open neighborhood of $c$ in  $\Ch(\partial_\infty \Delta)$ with respect to the cone topology. Then $V_c = U_{x,r,c}$ for some point $x \in \Delta$ and $r >0$. Moreover, it is also enough to assume that $x$ is a special vertex in $\Delta$. Then for each of the chambers $c' \subset V_c=U_{x,r,c}$, consider the Weyl cone $Q(x,c') \subset \Delta$ with base point $x$ and ideal chamber $c'$. Since $r >0$, we are dealing with a simplicial complex, and $V_c=U_{x,r,c}$ is open, then all of the cones $Q(x,c')$ must have in common at least the same chamber $C$ of $ Q(x,c) \subset \Delta$ having $x$ as a special vertex. 

Consider now the $\sigma'$-cone $Q(x, \sigma') \subset Q(x,c') \subset \Delta$ with base point $x$ and ideal simplex $\sigma' \subset c'$ of type $I$, for each of the ideal chambers $c' \in U_{x,r,c}$. Then by the above, the $\sigma'$-cones $Q(x,\sigma')$ will also have in common at least a face of $C$ having $x$ as a vertex. This shows that $V_\sigma$ is indeed $U(x,r', \sigma)$, for some $0< r' \leq r$, and thus $V_\sigma$ is a standard open neighborhood of $\sigma$ in $I(\partial_\infty \Delta)$.
\end{proof}

\subsection{Strongly $I$--regular hyperbolic elements}
\label{subsec::I_regular_hyp_elem}

In this section we generalize the notions of regular and strongly regular hyperbolic elements used in \cite[Section 2.1]{CaCi} (and the references therein) when the \cat space is a  Euclidean building $X$.

\begin{definition}
\label{def::regular}
Let $X$ be a Euclidean building of dimension $n \geq 1$, $\gamma$ a hyperbolic automorphism of $X$, $(W,S)$ the associated irreducible finite Coxeter system, and $I \subsetneq S$. Let 
$$\Min(\gamma):=\{x \in X \; | \; d(x, \gamma(x))=|\gamma| \},$$ 
where $|\gamma|$ denotes the translation length of $\gamma$, i.e. $|\gamma|:= \inf\{ d(x, \gamma(x))\; \vert \; x \in X\}$. Since $X$ is a simplicial complex the infimum is always attained. We say $\gamma$ is an \textbf{$I$-regular hyperbolic} if $\Min(\gamma)$ contains a flat of dimension $n-\vert I\vert$. We call $\gamma$ \textbf{strongly $I$-regular hyperbolic} if moreover the two endpoints of some (and hence all) of its translation axes lie, respectively, in the interior of two opposite ideal simplices of type $I$ of the spherical building at infinity $\partial_\infty X$. Here we abuse the notation, and use the type $I$ with respect to the walls in the affine Weyl group $(W^{\text{\tiny{aff}}},S^{\text{\tiny{aff}}})$ of $X$.
\end{definition}

Notice, when $I =\{\emptyset\}$, a flat of dimension $n-\vert \emptyset \vert$ is an apartment of $X$. Still, for $I =\{\emptyset\}$, our Definition \ref{def::regular} of $\emptyset$-regular does not imply regular in the sense of \cite[Section 2.1, Definition 2.1]{CaCi}. This is because in the latter definition, the set $\Min(\gamma)$ contains a unique apartment. Yet, the set $\Min(\gamma)$ from Definition \ref{def::regular} might contain more apartments.  In any case, \cite[Section 2.1, Definition 2.1]{CaCi} implies Definition \ref{def::regular}. Thus, strongly regular hyperbolic in the sense of \cite[Section 2.1]{CaCi} is strongly $\emptyset$-regular hyperbolic in the sense of  Definition \ref{def::regular}.

\begin{definition}
Let $X$ be a  Euclidean building, $(W,S)$ the associated irreducible spherical Coxeter system, and $I \subsetneq S$. A geodesic line $\ell$ is called \textbf{strongly $I$-regular} if its endpoints lie, respectively, in the interior of two opposite ideal simplices of type $I$ of the spherical building at infinity $\partial_\infty X$. (Here we abuse the notation and by the type $I$ we mean the walls corresponding to the affine Weyl group $(W^{\text{\tiny{aff}}},S^{\text{\tiny{aff}}})$.)
\end{definition}

\begin{remark}
By \cite[Part II, Prop. 6.2(3)]{BH99} we know that $\Min(\gamma)$ is a closed and convex subset of $X$ (i.e. convex means that if two points are in $\Min(\gamma)$ then also is the geodesic between them). Moreover, $\Min(\gamma)$ is $\gamma$-invariant, i.e. $\gamma(\Min(\gamma))=\Min(\gamma)$. The structure of $\Min(\gamma)$, for $\gamma$ hyperbolic, is given in \cite[Part II, Prop. 6.8]{BH99}.
\end{remark}

\begin{definition}
Let $X$ be a Euclidean building of dimension $n \geq 1$, and $\gamma$ a hyperbolic automorphism of $X$. We define the ideal boundary of the set $\Min(\gamma)$, that is denoted by $\partial_\infty \Min(\gamma)$, to be the set of all equivalence classes of geodesic rays that are contained in $\Min(\gamma)$. Two geodesic rays are in the same equivalent class (i.e. they determine the same point in $\partial_\infty \Min(\gamma)$) if they are asymptotic (see \cite[Part II, Def. 8.1]{BH99}). 
\end{definition}

Notice that $\partial_\infty \Min(\gamma)$ is a subset of $\partial_\infty X$. Moreover, from $\Min(\gamma)$ being convex we have that for any fixed point $p \in \Min(\gamma)$, there is a geodesic ray with basepoint $p$ and endpoint $\xi$ that is entirely contained in $\Min(\gamma)$, for every $\xi \in \partial_\infty \Min(\gamma)$ (see \cite[Part II, Prop. 8.2, Lem. 8.3]{BH99}). Using just the definitions, and the fact that $\gamma$ acts like a translation, thus sending the geodesic rays in $ \Min(\gamma)$ to parallel (i.e. asymptotic) geodesic rays, one can deduce that $\gamma$ fixes pointwise the ideal boundary $\partial_\infty \Min(\gamma)$.
 
\begin{lemma}
Let $X$ be a Euclidean building of dimension $n \geq 1$, and $\gamma$ a hyperbolic automorphism of $X$. Then for every $\xi \in \partial_\infty \Min(\gamma)$ we have that $\gamma(\xi)=\xi$.
\end{lemma}
\begin{proof}
First notice that the translation length $\vert \gamma  \vert$ of $\gamma$ is strictly positive, $\vert \gamma  \vert >0$. 
Fix a point $p \in \Min(\gamma)$. Take $\xi \in \partial_\infty \Min(\gamma)$ and consider the geodesic ray $\sigma:[0,\infty) \to \Min(\gamma)$ from $p$ to $\xi$, that is also denoted by $[p,\xi)$. This exists by the remark made just above. Then, since $\gamma$ leaves invariant $\Min(\gamma)$, we have that $\gamma([p,\xi))= [\gamma(p), \gamma(\xi))$ is still a geodesic ray in $\Min(\gamma)$, this is exaclty $\gamma\circ \sigma : [0,\infty) \to \Min(\gamma)$. But the two geodesic rays $\sigma$ and $\gamma\circ \sigma$ are asymptotic, as $d(\sigma(t), \gamma(\sigma(t)))= \vert \gamma\vert$ for every $t \in [0,\infty)$. The conclusion follows since $\xi=\gamma(\xi)$.
\end{proof}

As in \cite[Section 2.1]{CaCi} we also record in the next lemmas some easy description of strongly $I$-regular hyperbolic automorphisms. 

\begin{lemma}(see \cite[Lemma 2.3]{CaCi})
\label{lem:CharSRH}
Let $X$ be a Euclidean building of dimension $n\geq1$, $\gamma \in \Aut(X)$ a type-preserving hyperbolic automorphism, and $I \subsetneq S$. If the two endpoints of some (and hence all) of the translation axes of $\gamma$ lie, respectively, in the interior of two opposite ideal simplices of type $I$ of the spherical building at infinity $\partial_\infty X$, then $\gamma$ is strongly $I$-regular hyperbolic. Moreover, if $\gamma$ is strongly $\emptyset$-regular hyperbolic, then $\Min(\gamma)$ is an apartment of $X$.
\end{lemma}

When $\gamma$ is strongly $\emptyset$-regular hyperbolic,  the apartment $\Min(\gamma)$ will henceforth be called the \textbf{translation apartment} of $\gamma$; this apartment is uniquely determined. 

\begin{proof}[Proof of Lemma~\ref{lem:CharSRH}]
Since $\gamma$ is hyperbolic, and since the endpoints of some $\gamma$-axis, say $\ell$, lie respectively in the interior of two opposite simplices $\sigma_+, \sigma_- \subset \partial_\infty X$, then the same holds for all $\gamma$-axes, because of asymptoticity. As $\gamma$ is type-preserving, we get that $\gamma$ fixes pointwise both simplices $\sigma_+, \sigma_-$ (see \cite[Part II, Prop. 6.8]{BH99}). 

By considering a point in $\ell$, and say $\sigma_+$, from the axioms of a building, one can construct a flat $F$ of dimension $n - \vert I\vert$ of $X$ containing $\ell$, $\sigma_+$ and $\sigma_-$ in its boundary. The flat $F$ is part of an apartment of $X$.  Now, because $\gamma$ is an automorphism of an affine building, and the translation axis $\ell$ is strongly $I$-regular, the flat $F$ is uniquely determined by $\ell$ and $\sigma_{\pm}$.  Thus, $F$ must be invariant under $\gamma$. In particular, $\gamma$ preserves $\partial_\infty F$, but fixes pointwise $\sigma_+,\sigma_-$ in $\partial_\infty F$. Therefore, $\gamma$ must act trivially on $\partial_\infty F$ (since $\gamma$ is type-preserving); in other words $\gamma$ acts by translations on $F$. This proves that $F \subseteq \Min(\gamma)$, and so $\gamma$ is $I$-regular hyperbolic, and by hypothesis, $\gamma$ is strongly $I$-regular. 

Suppose now $\gamma$ is moreover strongly $\emptyset$-regular. Then $F$ is an apartment $A$ of $X$. Since any translation axis $\ell'$ of $\gamma$ has the same endpoints as $\ell$, and those lie in the interior of the $c_+,c_-$, respectively, then $\ell'$ must be entirely contained in $A$. Thus $A = \Min(\gamma)$. 
\end{proof}

\begin{remark}
Given a locally finite thick Euclidean building $X$, notice that any of the type-preserving hyperbolic elements of $\Aut(X)$ is strongly $I$-regular, for some $I \subsetneq S$, where $(W,S)$ is the associated irreducible spherical Coxeter system with $X$. 
\end{remark}

\begin{lemma}(see \cite[Lemma 2.5]{CaCi})\label{lem:LocalCrit:SRL}
Let $X$ be a Euclidean building, $I \subsetneq S$, and $\ell$ be a bi-infinite geodesic line in $X$. If $\ell$ contains two special vertices that are contained in the respective interiors of two and opposite $I$-sectors  in a common apartment of $X$, then $\ell$ is strongly $I$-regular. Conversely, if $\ell$ is strongly $I$-regular and contains three special vertices, then at least two of them must be contained in the interiors of two opposite $I$-sectors of a common apartment. 
\end{lemma}

\begin{proof}
Let $v, v' \in \ell$ be special vertices, respectively contained in the interiors of two opposite $I$-sectors $s,s'$ with the same base point, and which are contained in some apartment $\mathcal{A}$ of $X$. Then $\mathcal{A}$ contains the geodesic segment $[v, v']$, which cannot be parallel to any of the boundary-walls of the closure of the $I$-sectors $s,s'$. Extend this geodesic segment $[v,v']$ to a bi-infinite geodesic line $\ell'$ in $\mathcal{A}$. Then $\ell'$ is contained in the union of the open sectors $s,s'$, and $\ell'$ is strongly $I$-regular. Moreover, $[v,v']\subset \ell \cap \ell'$.

Let $\mathcal{B}$ be an apartment of $X$ containing $\ell$. Then both $\mathcal{A}$ and $\mathcal{B}$ contain the special vertices $v,v'$, that are in particular simplices. Thus, by the last axiom of the definition of a building, there exists an isomorphism between $\mathcal{A}$ and $\mathcal{B}$ which fixes $v,v'$ pointwise, thus also the unique geodesic $[v,v']$ between them. In particular, $\ell'$ is sent to $\ell$, which makes $\ell$ strongly $I$-regular as well.

The converse statement is straightforward. 
\end{proof}

\subsection{Existence of strongly $I$-regular elements}
\label{subsec::existence_strongly_regular}
In this section we generelize the existence of strongly regular hyperbolic elements from \cite[Section 2.2]{CaCi} to the strongly $I$-regular hyperbolic ones. But first we need some preparation. Although the results and the proofs  in this section about strongly $I$-regular elements are very similar with the ones in \cite{CaCi} for strongly regular elements, we rewrite them for the reader convenience.  The convergence in $\partial_\infty \Delta$ is to be understood with respect to the cone topology, which turns $\Delta \cup \partial_\infty \Delta$ into a compact space (see~\cite[Chap.II.8]{BH99}).

\begin{lemma}(See also \cite[Lemma 2.6]{CaCi})
\label{existance_reg_element}
Let $(W_{aff},S_{aff})$ be an irreducible Euclidean Coxeter system, $\mathcal{A}$ be the associated Euclidean Coxeter complex, $(W,S)$ the associated irreducible spherical Coxeter system, and $I \subsetneq S$. Then $W_{aff}$ contains strongly $I$-regular hyperbolic elements.
\end{lemma}

\begin{proof}
Recall that $W_{aff}$ acts transitively on the set of special vertices of $(W_{aff},S_{aff})$ that are in the same equivalence class (i.e. of the same type). In particular, the translation group of $W_{aff}$ acts simple-transitively on the above-mentioned set of special vertices. Now let a special vertex $x$ in $\mathcal{A}$ and consider two opposite $I$-sectors of $\mathcal{A}$ with the same base point $x$. By the construction of the corresponding Coxeter complex associated with $(W_{aff},S_{aff})$, there are two special vertices  $v$ and $v'$ in the same equivalence class which are contained in the interiors of the two  chosen opposite $I$-sectors, respectively. Then the unique geodesic line $\ell$ through $v$ and $v'$ is contained in the flat generated by the two chosen opposite $I$-sectors, and so $\ell$ is strongly $I$-regular. Then the unique translation $t$ mapping $v$ to $v'$ is an element of $W_{aff}$ having $\ell$ as a translation axis. Thus $t$ is strongly $I$-regular hyperbolic by Lemma~\ref{lem:CharSRH}. 
\end{proof} 

The following basic construction is a key step in proving the existence of strongly $I$-regular elements in groups acting on Euclidean buildings. 

\begin{lemma}(See also \cite[Lemma 2.7]{CaCi})
\label{lem:basic-const}
Let $X$ be a locally finite thick Euclidean building, $(W,S)$ the associated irreducible spherical Coxeter system, $I \subsetneq S$, and $G$ a type-preserving subgroup of $\Aut(X)$, with cocompact action on $X$. Let $\rho \colon \RR \to X$ be a geodesic map such that $\ell: = \rho(\RR)$ is a strongly $I$-regular geodesic line, and $C >0$ be such that $\rho(nC)$ is a special vertex of a given type, for all $n \in \NN$. 

Then there exist a sequence $\{g_{n}\}_{n \geq 0} $ in $ G$ and  an increasing sequence $f(n)$ of positive integers such that for every $m \geq 0$ and every $r \in [-m, m]$, the sequence $n \mapsto g_n \circ \rho({f(n)C} + r)$ is constant on $\{n \; | \; n \geq m\}$. Moreover, the map 
$$\rho' \colon \RR \to X :  r \mapsto  \lim\limits_{n \to \infty} g_n \circ \rho({f(n)}C +r)$$
is a geodesic line, and its image $\rho'(\RR)$ is a strongly $I$-regular geodesic line. 
\end{lemma}

\begin{proof}
%
Set $t_n = nC$ and $x_n = \rho(nC)$. 
Since $G$ is type-preserving and with cocompact action on $X$, up to replacing $\{x_{n}\}_{n \geq 0}$ by a subsequence, we can find a sequence $\{g_{n}\}_{n \geq 0} \subset G$ such that for every $n \geq 0$, $g_{n}(x_{n})$ is the special vertex $v_0$ in a fixed fundamental domain for the $G$-action on  $X$. 

Since $X$ is locally finite, any ball around $v_0$ contains finitely many special vertices. On the other hand, for any $n \geq m \geq 0$, the geodesic segment $g_n \circ \rho([t_n - m, t_n +m])$, which is of length $2m$, contains $\lfloor \frac{2m} C \rfloor$ special vertices, of the same type, including $v_0$. Therefore, we may extract subsequences to ensure that for every $ m \geq 0$ and every $r \in [-m, m]$, the sequence $n \mapsto  g_n \circ \rho (t_n +r)$  is constant on $\{n \; | \; n \geq m\}$. By construction, the map $\rho' \colon r \mapsto \lim\limits_{n\to \infty}   g_n \circ \rho (t_n +r)$ is well defined. Moreover it is a limit of a sequence of geodesic maps converging uniformly on compact sets, and thus $\rho'$ is a geodesic. That $\ell' = \rho'(\RR)$ is strongly $I$-regular follows  from Lemma~\ref{lem:LocalCrit:SRL}.
\end{proof}

Lemma 2.8 from \cite{CaCi} remains the same.

The next proposition is again a technical version of Theorem~\ref{thm:ExistenceStronglyReg_1}, and it is stated separately to ease future references. 

\begin{proposition}(See also \cite[Prop. 2.9]{CaCi})
\label{prop:TechSRH}
Let $X$ be a locally finite thick Euclidean building, $(W,S)$ the associated irreducible spherical Coxeter system, $I \subsetneq S$, and $G$ a type-preserving subgroup of $\Aut(X)$, with cocompact action on $X$. Let $\rho \colon \RR \to X$ be a geodesic map such that $\ell := \rho(\RR)$ is a strongly $I$-regular geodesic line, and let $C >0$ be such that $\rho(nC)$ is a special vertex of a given type for all $n \in \NN$. 

Then there is an increasing sequence $f(n)$ of positive integers such that, for any $n> m > 0$, there is a strongly $I$-regular hyperbolic element $h_{m, n} \in G$ which has a translation axis containing the geodesic segment $[\rho(f(m)C), \rho(f(n)C)]$. 
\end{proposition}

\begin{proof}
Let $\{g_n\}_{n\geq 0} \subset G$ and $f(n) \in \NN$ be the sequences afforded by  Lemma~\ref{lem:basic-const}. Fix $m > 0$. Then by Lemma~\ref{lem:basic-const}, we know that for every $r \in [-m, m]$, the map $n \mapsto g_n \circ \rho({f(n)C} + r)$ is constant on $\{n \; \vert \; n \geq m\}$. In particular, for every $n>m$, the element $h_{m, n} = g_m^{-1} g_n \in G$ has the property that 
$$h_{m, n} \circ \rho(f(n)C +r) = \rho(f(m)C+r)$$
for any $r \in [-m, m]$. Therefore \cite[Lemma 2.8]{CaCi} ensures that $h_{m, n}$ is hyperbolic and has a translation axis passing through $\rho(f(m)C)$ and $ \rho(f(n)C)$. This axis must therefore be strongly $I$-regular by Lemma~\ref{lem:LocalCrit:SRL}, and we conclude that $h_{m, n}$ is strongly  $I$-regular hyperbolic in view of Lemma~\ref{lem:CharSRH}. 
\end{proof}


\begin{proof}[Proof of Theorem \ref{thm:ExistenceStronglyReg_1}]
Since $X$ is thick, the subgroup of $\Aut(X)$ consisting of type-preserving automorphisms is of finite index. In particular, $G$ has a finite index subgroup acting by type-preserving automorphisms, and whose action remains cocompact. There is thus no loss of generality in assuming that the $G$-action is type-preserving.

Let $\mathcal{A}$ be an apartment in $X$. By Lemma~\ref{existance_reg_element}, the affine Weyl group of $X$ acting on $\mathcal{A}$ contains strongly $I$-regular hyperbolic elements. Any point, and so also any special vertex, of the apartment $\mathcal{A}$ belongs to a translation axis of such a fixed hyperbolic element. It follows there exist a geodesic map $\rho \colon \RR \to \mathcal{A}$ and some $C>0$ such that $\ell = \rho(\RR)$ is strongly $I$-regular and that $\rho(nC)$ is a special vertex for all $n \in \ZZ$. The desired conclusion now follows from Proposition~\ref{prop:TechSRH}.
\end{proof}

\subsection{Dynamics of strongly $I$-regular elements}
\label{subsect:dynamics-reg}

Since the visual boundary of $\Min(g)$ of a strongly $I$-regular element $g$ of $\Aut(\Delta)$  is generally not reduced to the visual boundary of a flat, or an affine apartment, the dynamics of such element $g$ on the ideal points outside $\partial_\infty \Min(g)$ is not very clear, and therefore, the same result as \cite[Prop. 2.10]{CaCi} cannot be expected. Still, if we consider points in $\partial_\infty \Delta$ with specific properties, then \cite[Prop. 2.10]{CaCi} will become true.

We recall the following proposition which is a key ingredient in this section.

\begin{proposition}(See \cite[ Prop. 4.106]{AB})
Let $c$ and $d$ be chambers in a spherical building. If $c$ and $d$ are not opposite, then there is a panel (i.e. facet) $p< c$ such that $c = \proj_p(d)$.
\end{proposition}


%

Recall that a residue is a set of chambers of a spherical/affine building with specific properties (see \cite[Section 5.3.1.]{AB}). In this section the residues are seen from the perspective of the finite Weyl group $(W,S)$, and not with respect to the walls of the affine Weyl group. Thus, two opposite simplices in $\partial_\infty \Delta$ might have different types. 

\begin{remark}
\label{rem::residues_proj}
Let $X$ be an irreducible spherical building of associated irreducible spherical Coxeter system $(W,S)$. Let $I,J \subsetneq S$, and $\mathcal{R}$, respectively $\mathcal{S}$, residue of type $I$, respectively $J$, in $X$. Then, by \cite[Lemma 5.29]{AB}, the set $\delta(\mathcal{R}, \mathcal{S})$ is a double coset $W_I w W_J$, for some $w \in W$. Also, by \cite[Prop. 5.36(2)]{AB},  there is a unique element $w_1$ in $W$ such that $w_1 = \min(\delta(\mathcal{R}, \mathcal{S}))$. Moreover, $w_1$ has minimal length in the double coset $W_I w W_J$ (see \cite[Section 5.3.2.]{AB}). Then $\proj_{\mathcal{R}}(\mathcal{S})$ is a residue of type $I \cap w_1Jw_1^{-1}$, where $I \cap w_1Jw_1^{-1}$ should be seen as the intersection of the associated reflections from $W$ in the corresponding Euclidean space associated with the finite Coxeter system $(W,S)$.  Notice the following cases:
\begin{enumerate}
\item 
If $I \cap w_1Jw_1^{-1}=I$, then $\proj_{\mathcal{R}}(\mathcal{S}) = \mathcal{R}$.
\item
If $I \cap w_1Jw_1^{-1} \neq \emptyset$, then $\proj_{\mathcal{R}}(\mathcal{S})$ will be a residue that consists of more than one chamber.
\item
If $I \cap w_1Jw_1^{-1} = \emptyset$, then $\proj_{\mathcal{R}}(\mathcal{S})$ will just be a chamber.
\end{enumerate}
\end{remark}

\begin{lemma}
\label{lem::separating_walls}
Let $X$ be an irreducible spherical building of associated irreducible spherical Coxeter system $(W,S)$. Let $I,J \subsetneq S$, and $\mathcal{R}$, respectively $\mathcal{S}$, residue of type $I$, respectively $J$, in $X$ such that $\mathcal{R} \cap \mathcal{S} = \emptyset$, seen as the intersection of two sets of chambers. Then for any chamber $c \in \mathcal{S}$, the set of all walls of the chamber $\proj_{\mathcal{R}}(c) \in \mathcal{R}$ sperating it from $c$ does not contain any wall in $\mathcal{R}$ corresponding to the elements of $W_I$.
\end{lemma}

\begin{proof}
Take some $c  \in \mathcal{S}$, and denote $f:=\proj_{\mathcal{R}}(c) \in \mathcal{R}$. Consider $c':=\proj_{\mathcal{S}}(f) \in \mathcal{S}$, and by \cite[Prop. 5.34]{AB} we have $d(c,f)= d(f,c')+d(c',c)$. Moreover, by a simple computation using the same \cite[Prop. 5.34]{AB}, one has that $\proj_{\mathcal{R}}(c')= f \in \mathcal{R}$, as otherwise $f' := \proj_{\mathcal{R}}(c') \neq f$ would have the property that $d(c,f') < d(c,f)$, giving a contradiction. Then, by \cite[Prop. 5.36]{AB} we know that 
\begin{equation}
\label{equ::proj_residue_dis}
\proj_{\mathcal{R}}(\mathcal{S}) = \{f \in \mathcal{R} \; \vert \; \text{there is } c\in \mathcal{S} \text{ with } d(f,c) = d(\mathcal{R}, \mathcal{S})\},
\end{equation}
where by definition \cite[Def. 5.35]{AB} $\proj_{\mathcal{R}}(\mathcal{S}):=\{\proj_{\mathcal{R}}(c) \; \vert \; c \in\mathcal{S}\}$.
Since $\proj_{\mathcal{R}}(c')= f$ and $c':=\proj_{\mathcal{S}}(f)$, equality (\ref{equ::proj_residue_dis}) guarantees that $d(f,c')=d(\mathcal{R}, \mathcal{S})$. Moreover, there is a unique element $w_1 \in W$ such that $w_1 = \min(\delta(\mathcal{R}, \mathcal{S}))$, and moreover $w_1$ has minimal length in the double coset $W_I w W_J = \delta(\mathcal{R}, \mathcal{S})$.  In particular, since  $\mathcal{R} \cap \mathcal{S} = \emptyset$, we have $w_1$ is not the trivial element  of $W$. This means that the set of all walls separating $f$ from $c'$, and thus from $c$, cannot contain any wall in $\mathcal{R}$ corresponding to the elements of $W_I$, as otherwise $w_1=\delta(f,c')$ would not be minimal in $W_I w_1 W_J$. 
 \end{proof}

In the next proposition, we use $\rho_{\mathcal{A},c_-}$ to denote the retraction from the ideal chamber $c_- \in \Ch(\partial_\infty \mathcal{A})$ to the apartment $\mathcal{A}$. We consider residues as sets of chambers.

\begin{proposition}
\label{prop::dynamics_I_regular}
Let $\Delta$ be a Euclidean building,  $(W,S)$ the associated irreducible spherical Coxeter system, $I \subsetneq S$, and $\gamma \in \Aut(\Delta)$ a type-preserving, strongly $I$-regular hyperbolic element.  Denote by $\sigma_+,\sigma_-$ the minimal ideal simplices containing the attracting and repelling  endpoints of $\gamma$ in their interior, respectively. Let $\mathcal{R}_{-}$ be the $I$-residue corresponding to $\sigma_-$.   Let $\xi \in \partial_\infty \Delta \setminus \partial_\infty \Min(\gamma)$, $\beta$ the minimal ideal simplex containing $\xi$ in its interior whose type is denoted by $J$, and $\mathcal{S}$ the $J$-residue corresponding to $\beta$. 

Then, under the assumption that $\proj_{\mathcal{R}_{-}}(\mathcal{S}) \cap \partial_\infty \Min(\gamma)$ contains at least one ideal chamber,  the limit $\lim\limits_{n\to \infty} \gamma^{n}(\xi)$ exists (in the cone topology), and coincides with $\rho_{\mathcal{A}',c_-}(\xi)$, for  any $c_- \in \proj_{\mathcal{R}_{-}}(\mathcal{S}) \cap \partial_\infty \Min(\gamma)$ and any apartment $\mathcal{A}'$ with $c_-, \sigma_+ \subset \partial_\infty \mathcal{A}'$.

\end{proposition}

\begin{proof}
Let $\xi \in \partial_\infty \Delta \setminus \partial_\infty \Min(\gamma)$. We are under the assumption that $\proj_{\mathcal{R}_{-}}(\mathcal{S}) \cap \partial_\infty \Min(\gamma)$ contains at least one ideal chamber. Let $\mathcal{R}_{+}$ be the $I$-residue corresponding to $\sigma_+$, thus $\mathcal{R}_{+}$ and $\mathcal{R}_{-}$ are opposite.

We have two cases to treat, when $\proj_{\mathcal{R}_{-}}(\mathcal{S})$ is just a chamber, and when $\proj_{\mathcal{R}_{-}}(\mathcal{S})$ is a residue with more than one chamber. The latter case is possible only when $I \cap w_1J w_1^{-1}\neq \emptyset$.

We also make the following remark. Let $f \in \mathcal{R}_{-} \cap \partial_\infty \Min(\gamma)$, and take $c:=\proj_{\mathcal{R}_{+}}(f)$, which is uniquely determined by $f$. Since $\gamma$ fixes pointwise $f$, by uniqueness of $c$, the same is true for $c$, i.e. $\gamma(c)=c$ pointwise. Consider now the combinatorial convex hull $\Gamma(f,c)$ in the spherical building $\partial_\infty \Delta$.  $\Gamma(f,c)$ is also unique, contains every minimal galery between $c$ and $f$, and is contained in any apartment of  $\partial_\infty \Delta$ having both $c,f$ as chambers (see \cite[Section 4.11.1]{AB}). Then again by uniqueness, we have that $\gamma(\Gamma(f,c))=\Gamma(f,c)$ pointwise. 

\textbf{Case 1: $\proj_{\mathcal{R}_{-}}(\mathcal{S})$ is just a chamber.} Take $c_-:= \proj_{\mathcal{R}_{-}}(\mathcal{S}) \subset \partial_\infty \Min(\gamma)$, and consider $c:=\proj_{\mathcal{R}_{+}}(c_-)$. Notice that $\sigma_+ \subset c$, but $c$ is not opposite $c_-$. Fix for what follows an apartment $\mathcal{A}$ of $\Delta$ such that $c_-,c \in \Ch(\partial_\infty \mathcal{A})$, and such that $\Min(\gamma) \cap \mathcal{A} \neq \emptyset$ (which is always true). Then $\Gamma(c_-,c) \subset \partial_\infty \mathcal{A}$, but it might be the case that $\gamma$ does not leave invariant $\mathcal{A}$, maybe just a convex part of it. 

Consider an apartment $\mathcal{B}$ of $\Delta$ whose boundary contains $\xi$ and $c_-$. Since $\xi \notin  \partial_\infty \Min(\gamma)$ we have $\xi \not \in c_-$. Moreover, notice that $\mathcal{B}$ contains an ideal chamber $b$ in its boundary such that $\xi \in b$ and $d(c_-,b)=d(\mathcal{R},\mathcal{S})$. 

Let $Q$ be a closed sector in $\mathcal{A} \cap \mathcal{B}$ with base point $x \in \Min(\gamma)$ and pointing towards $c_-$. Consider the combinatorial  convex hull $\Gamma(c_-,c,x)$ in $\Delta$, that in entirely contained in $\mathcal{A}$. In particular, $\Gamma(c_-,c) \subseteq \partial_\infty \Gamma(c_-,c,x)$.

Since $\Min(\gamma)$ is convex and $c_- \subset \partial_\infty \Min(\gamma)$, we have that $Q \subset \Gamma(c_-,c,x) \subset \Min(\gamma)$. Let moreover $p \in Q$ be a base point contained in the interior of $Q$. Since $\mathcal{B}$ is convex, the geodesic ray $[p, \xi)$ is entirely contained in $\mathcal{B}$. In particular, the ray $[p, \xi)$ is not entirely contained in $Q$ and hence there is some $q \in Q$ such that $[p, \xi) \cap Q = [p, q]$. Moreover, by Lemma \ref{lem::separating_walls}, the set of all walls of $c_-$ separating it from $b$ does not contain any of the walls correspondig to the elements in $W_I$. So the point $q$ lies in at least one of the walls of $Q$ that are different from the ones corresponding to elements of $W_I$. But $q$ is not contained in any of the walls corresponding to elements of $W_I$. 

Now, because $p$ lies in the interior of $Q$, it follows that the segment $[p, q]$ is of positive length. Therefore, in the apartment $\mathcal{A}$, the geodesic segment $[p, q]$ can be extended uniquely to a geodesic ray emanating from $p$; in other words, there is a unique boundary point $\eta \in \partial_\infty \mathcal{A}$ such that the segment $[p, q]$ is contained in the ray $[p, \eta)$. 

\medskip
We claim that $\lim\limits_{n \to \infty} \gamma^n(\xi) = \eta$, and $\eta = \rho_{\mathcal{A}, c_-}(\xi)$, and that $\eta \in \Gamma(c_-,c) \subset \partial_\infty \Min(\gamma)$.

Indeed, that $\eta \in \Gamma(c_-,c) \subset \partial_\infty \Min(\gamma)$ follows easily from the fact that 
$$[p,q] \subset Q \subset \Gamma(c_-,c,x) \subset \Min(\gamma),$$ 
convexity of $\Gamma(c_-,c,x) \subset \mathcal{A}$, and the fact that $q$ lies in at least one of the walls of $Q$ that are different from the ones corresponding to elements of $W_I$. 

For the rest of the claim we proceed as follows. Since $\gamma$ is strongly $I$-regular with repelling simplex $\sigma_-$, and $ \Gamma(c_-,c,x) \subset \mathcal{A} \cap \Min(\gamma)$, we get that $\gamma(\Gamma(c_-,c,x))=\Gamma(c_-,c,x)$ setwise. Then $ Q \subset \gamma^n(Q) \subset \Gamma(c_-,c,x)$ for all $n >0$. Moreover $[\gamma^n(p), \gamma^n(q)] = [\gamma^n(p), \gamma^n(\xi)) \cap \gamma^n(Q)$. Because $\gamma^n$ acts on $\Gamma(c_-,c,x)$ as a Euclidean translation, it follows that $[\gamma^n(p), \gamma^n(q)] = \gamma^n([p, q]) $ is parallel to $[p, q]$ (in the Euclidean sense), or even more, it can be the case that they both lie on a  translation axis of $\gamma$ in $\Gamma(c_-,c,x)$. On the other hand, the apartment $\gamma^n(\mathcal{B})$ contains both the ray $[p, \gamma^n(\xi))$ and the ray $[\gamma^n(p), \gamma^n(\xi))$, which are parallel since they have the same endpoint $\gamma^n(\xi)$. It follows that $[p, \gamma^n(\xi))$ contains $[p, q]$. 

Our next important observation is the following. Recall that $\gamma$ is strongly $I$-regular and $q$ does not lie in any of the walls of $Q$ parallel with the ones corresponding to the elements of  $W_I$. Then, for any boundary wall $M$ of $Q$ that is not parallel with same wall from $W_I$, the distance between $M$ and $\gamma^n(M)$ grows linearly with $n$. This implies that the length of the geodesic segment $[p, \gamma^n(\xi)) \cap \gamma^n(Q)$ tends to infinity with $n$. In particular, so does the length of $[p, \gamma^n(\xi)) \cap \mathcal{A}$.  Notice that if $q$ had lied in one of the walls corresponding to elements in $W_I$, then the segment $[p, \gamma^n(\xi)) \cap \gamma^n(Q)$ would have had constant distance. 

We have thus shown that the geodesic ray $[p, \gamma^n(\xi))$ contains $[p, q]$ for all $n \geq 0$ and also a subsegment of $\mathcal{A}$ whose length tends to infinity with $n$. By the definition of $\eta$, this implies that  $[p, \gamma^n(\xi)) \cap [p, \eta) $ is a segment whose length also tends to infinity with $n$. In particular we have $\lim\limits_{n \to \infty} \gamma^n(\xi) = \eta$.  Also, by the definition of the retraction $\rho_{\mathcal{A}, c_-}$ and by the construction above we have $\rho_{A, c_-}(\xi)=\eta$. The claim stands proven. 

To finish the proof of the proposition in the considered case, we notice that the choice of the apartment $\mathcal{A}$ does not matter, because for any other apartment $\mathcal{A}'$ of $\Delta$ such that $c_-,c \in \Ch(\partial_\infty \mathcal{A}')$, we have that $\mathcal{A}' \cap \Min(\gamma) \neq \emptyset$ and $\Gamma(c_-,c) \subset \partial_\infty \mathcal{A}'$. 

\medskip
\textbf{Case 2: $\proj_{\mathcal{R}_{-}}(\mathcal{S})$ is a residue with more than one chamber.}  By Remark \ref{rem::residues_proj}, this will be the case only when $I \cap w_1Jw_1^{-1} \neq \emptyset$.

Take $c_- \in \proj_{\mathcal{R}_{-}}(\mathcal{S}) \cap \partial_\infty \Min(\gamma)$, and consider $c:=\proj_{\mathcal{R}_{+}}(c_-)$. Proceed as in Case 1, for $c_-, c$ and any apartment $\mathcal{A}$ of $\Delta$ such that $c_-,c \in \Ch(\partial_\infty \mathcal{A})$. As $I \cap w_1Jw_1^{-1} \neq \emptyset$ represent reflections with respect to the walls of the affine building $\Delta$, we have that $\mathcal{S}$ contains a sub-residue laying in the same set of walls in $\Delta$ as a sub-residue of $\mathcal{R}_{-}$. Then we get that the segment $[p,q]$ is parallel with at least one of the walls corresponding to the elements of $W_I$. Still, $q$ does not lie in any of the walls of $Q$ corresponding to the elements of $W_I$.  Then at the end of the day, one can notice that $\eta$ is a precise point in the closure of the combinatorial  convex hull $\Gamma(c_-,c)$ in $\partial_\infty \Delta$, and $\eta$ does not depend on the choice of $c_-$ in $\proj_{\mathcal{R}_{-}}(\mathcal{S})$ or of the apartment $\mathcal{A}$. 

If one takes $c_- \in \proj_{\mathcal{R}_{-}}(\mathcal{S})$ such that $c_- \notin \partial_\infty \Min(\gamma)$, then for $c:=\proj_{\mathcal{R}_{+}}(c_-)$ we have that $\{\gamma^n(c)\; \vert \; n \in \NN\}$ is still a set of chambers in $\mathcal{R}_{+}$, but the their dynamics is chaotic with no clear convergent pattern.
\end{proof}

\begin{remark}
For a $\xi \in \partial_\infty \Delta \setminus \partial_\infty \Min(\gamma)$ that does not satisfy the conditions of Proposition \ref{prop::dynamics_I_regular}, i.e. $\proj_{\mathcal{R}_{-}}(\mathcal{S}) \cap \Ch(\partial_\infty \Min(\gamma)) = \emptyset$, it might be the case that the set $\{\gamma^{n}(\xi) \; \vert \; n\in \NN\}$ will rotate around either the repelling endpoint of $\gamma$, or some translation axis of $\gamma$, without a clear convergent pattern.
\end{remark}

\begin{remark} Since two disjoint, opposite open sectors (not necessarily with a common base point) that are translated along an axis that is parallel to some of the boundary walls of both of the sectors might still have an empty intersection after applying that translation, the proof of \cite[Prop. 2.11]{CaCi} cannot work for strongly $I$-regular hyperbolic elements.
\end{remark}

\begin{remark}
For split semi-simple algebraic groups $G$ over non-Archimedean local fields, if one considers strongly $I$-regular hyperbolic elements $\gamma$ from a maximal split torus of $G$, the dynamical behaviour of such $\gamma$, in the sense of Propositon \ref{prop::dynamics_I_regular}, is much nicer as both of the residues $\mathcal{R}_{-}, \mathcal{R}_{+}$ will be contained in $\partial_\infty \Min(\gamma)$.
\end{remark}

\section{Applications to strongly $I$-regular elements}
\label{sec::applications}
\subsection{Some intermediate results}

Let $\Delta$ be a locally finite Euclidean building,  $(W,S)$ the associated irreducible spherical Coxeter system, $I \subsetneq S$, and $\gamma \in \Aut(\Delta)$ a type-preserving, strongly $I$-regular hyperbolic element. Consider the attracting and repelling endpoints $\xi_+,\xi_-$ of $\gamma$, and the minimal ideal simplices $\sigma_+,\sigma_- \subset \partial_\infty \Delta$ that contain those in their interior, respectively. Notice $\partial_\infty \Min(\gamma)$ is contained in the combinatorial  convex hull  $Conv_{\partial_\infty  \Delta} (\st_{\partial_\infty  \Delta}(\sigma_+),\st_{\partial_\infty  \Delta}(\sigma_-) )$. In general, that inclusion might be strict.

Fix a geodesic line $\ell$ in $\mathcal{A}$ such that $\ell$ is a translation axis of $\gamma$. Thus, the endpoints  $\xi_+,\xi_-$ of $\ell$ are in the interior of $\sigma_+$, $\sigma_-$, repsectively. We can write $\ell=[\xi_-,\xi_+]$. We want to understand the topology on 
$$\Opp(\sigma_+) :=\{\sigma \subset \partial_\infty \Delta \; \vert \; \sigma \text{ opposite } \sigma_+ \}.$$
Notice that $\sigma_- \in \Opp(\sigma_+)$, also $\Opp(\sigma_+)$ is in natural bijection with $\Opp(\xi_+)$, and $\xi_- \in \Opp(\xi_+)$.

\begin{remark}
We claim $\Opp(\sigma_+)\setminus \sigma_-$ is not a subset of $\partial_\infty \Min(\gamma)$.  Recall the convexity properties of $\Min(\gamma)$ from \cite[Chapter II, Prop. 6.2, 6.8]{BH99}: $\Min(\gamma)$ is a closed convex set in $\Delta$, and $\Min(\gamma)$ is the union of all translation axes of $\gamma$. From those properties one can see that $\partial_\infty \Min(\gamma)$ is contained in $Conv_{\partial_\infty  \Delta} (\st_{\partial_\infty  \Delta}(\sigma_+),\st_{\partial_\infty  \Delta}(\sigma_-) )$. In general, that inclusion might be strict. Thus, if there is a geodesic line $\ell$ in $\Min(\gamma)$ with one endpoint $\xi_+$ and which intersects, or it is parallel to, a translation axis of $\gamma$ in a geodesic ray with endpoint $\xi_+$, then the other endpoint of the geodesic line $\ell$ can only be $\xi_-$. Using that, one can deduce that appart from $\sigma_-$ no element from $\Opp(\sigma_+)\setminus \sigma_-$ will be fixed by $\gamma$, and so $\Opp(\sigma_+)\setminus \sigma_- \cap \partial_\infty \Min(\gamma) =\emptyset$.
\end{remark}

\begin{remark}
Because the boundary $\partial_\infty \Delta$ is a compact space with respect to the cone topology on $\Delta \cup \partial_\infty \Delta$, and since the Grassmannian $G/P_{I_{\sigma_+}}$ is also compact, one can deduce that $I_{\sigma_+}(\partial_\infty \Delta)$, which is isomorphic to $G/P_{I_{\sigma_+}}$, is also compact, and thus closed subset of $\partial_\infty \Delta$. 
\end{remark}

\begin{lemma}
\label{lemma_10}
Let $\sigma$ be an ideal simplex in $\partial_\infty \Delta$ and consider two apartments $\mathcal{B}_1,\mathcal{B}_2$ of $\Delta$ such that $\partial_\infty \mathcal{B}_1 \cap \partial_\infty \mathcal{B}_2$ contains $\st(\sigma,\partial_\infty \mathcal{B}_1)$. Then there is a common $\sigma$-cone in both $\mathcal{B}_1$ and $\mathcal{B}_2$.
\end{lemma}
\begin{proof}
For any ideal chamber $c$ in $\st(\sigma,\partial_\infty \mathcal{B}_1)$, the apartments $\mathcal{B}_1, \mathcal{B}_2$ have a sector in common corresponding to $c$; denote it by $Q_{x_c}(c)$, with $x_c$ its base point.  Now, since all $Q_{x_c}(c)$ are in $\mathcal{B}_1 \cap \mathcal{B}_2$, for any $c \in \st(\sigma,\partial_\infty \mathcal{B}_1)$, by working in a Euclidean space, it is then easy to pick a point $x \in \mathcal{B}_1 \cap \mathcal{B}_2$ and for each $c \in \st(\sigma,\partial \mathcal{B}_1)$ to draw a sector with base point $x$ and pointing towards $c$. Then we get our $\sigma$-cone in both $\mathcal{B}_1$ and $\mathcal{B}_2$ having base point $x$. The lemma is proven.
\end{proof}

For the rest of the section, and for simplicity, \textit{we will suppose that the apartment $\mathcal{A}$ is a subset of $\Min(\gamma)$}. This is because for the applications that we have in mind this will be the case, and as mentioned before it is easier to handle the dynamics of such strongly $I$-regular hyperbolic elements $\gamma$.

In Lemma \ref{lem::lem_36} we have proved that an apartment $\mathcal{B}$ of $\Delta$, containing $\sigma_+,\sigma_-$ in its ideal boundary, uniquely determines and is uniquely determined by $\st(\sigma_+, \partial_{\infty}\mathcal{B})$ and $\st(\sigma_-, \partial_{\infty}\mathcal{B})$. Then take $\sigma'_- \in \Opp(\sigma_+) \setminus \sigma_-$. To $\st(\sigma_+, \partial_{\infty}\mathcal{A})$ there is associated a unique $\st(\sigma'_-, \partial_{\infty}\mathcal{A}') \subset \st_{\partial\Delta}(\sigma_-)$ such that $\st(\sigma'_-, \partial_{\infty}\mathcal{A}')$ and $\st(\sigma_+, \partial_{\infty}\mathcal{A})$ are in the same apartment; this apartment is uniquely determined and denote it by $\mathcal{A}_{\sigma_{+},\sigma'_{-}}$. Now we apply Lemma \ref{lemma_10} to the apartments $\mathcal{A}$ and $\mathcal{A}_{\sigma_{+},\sigma'_{-}}$.  And so we deduce that $\ell \cap \mathcal{A}_{\sigma_{+},\sigma'_{-}}$ is a ray, say $[y_{\sigma'_-}, \xi_+)$, and denote by $\ell_{\sigma'_-}$ the geodesic line in $\mathcal{A}_{\sigma_{+},\sigma'_{-}}$ having endpoints $\xi'_{-}, \xi_+$ and containing  $[y_{\sigma'_-}, \xi_+)$. Notice, $\xi'_{-}$ is opposite $\xi_+$ and $\xi'_{-} \in \sigma'_-$. Since $\xi'_{-}$ and $\sigma'_-$ are not in $\partial_\infty \Min(\gamma)$, when we apply $\gamma$ to the line $\ell_{\sigma'_-}$, we get $\xi'_{-} \neq \gamma(\xi'_{-})$ and $\sigma'_- \neq \gamma(\sigma'_-)$.  But still, $ \gamma^n(\sigma'_-) \in \Opp(\sigma_+)$ and $ \gamma^n(y_{\sigma'_-}) \in \ell$ for every $n  \in \ZZ$. 

\begin{remark}
\label{rem::rem3.31}
For $\gamma$ a strongly $I$-regular hyperbolic element such that there is an apartment $\mathcal{A} \subset \Min(\gamma)$ we always have $\gamma(\Opp(\sigma_+))= \Opp(\sigma_+)$. This is also true without the assumption on the existence of the apartment $\mathcal{A}$ in $\Min(\gamma)$.
\end{remark}

Thus, for every $\sigma'_- \in \Opp(\xi_+)$, we can view some of the open neighborhoods of $\xi'_-$ with respect to the cone topology from the points of the fixed $\gamma$-translation axis $\ell$. This is possible because the cone topology does not depend on the chosen base point (see Proposition \ref{prop::independence_base_point}). Pick a point $x \in \ell$, and $-r \in (x,\xi_-)$. Then we take 
$$V_{x,-r}:=\{\xi'_- \in \Opp(\xi_+) \; \vert \; [-r,x] \subset [x,\xi'_-) \cap \ell \}.$$

\begin{lemma}
\label{lem::lem11}
For every $x \in \ell$ and $-r_1\in (x,\xi_-)$ we have $\gamma(V_{x,-r_1})= V_{\gamma(x),\gamma(-r_1)}$, and $V_{x,-r_2} \subset V_{\gamma(x),\gamma(-r_1)}$ for every $-r_2 \in  (\gamma(-r_1),\xi_-) \cap (x,\xi_-)$.
\end{lemma}
\begin{proof}
The first part just follows from the definition.
The second part goes as follows. Since $\gamma(x) \in (x, \xi_+)$ and $\gamma(-r_1) \in (-r_2, \xi_+)$, then the distance between $\gamma(x)$ and $-r_2$ is larger than the distance from $\gamma(x)$ to $\gamma(-r_1)$, and $[-r_2, x] \subset (\xi_-,\gamma(x))$. Thus for any $\xi'_- \in V_{x,-r_2}$,  the geodesic line $\ell_{\xi'_-}$ will contain the ray $(-r_2, \xi_+)$ and thus by the definitions, $\xi'_-$ will be in $V_{\gamma(x),\gamma(-r_1)}$.
\end{proof}

\begin{lemma}
\label{lem::lem12}
Let $x \in \ell$ and $-r_1\in (x,\xi_-)$. Take $\xi'_- \in \Opp(\xi_+)$. Then there is $N >0$ such that $\xi'_-\in \gamma^n(V_{x,-r_1})$ for every $n >N$.
\end{lemma}
\begin{proof}
We use the same notation as just before Remark \ref{rem::rem3.31}.
 Pick some $y \in \ell \cap \ell_{\sigma'_-}$, where $\xi'_- \in \sigma'_-$.
Since  $\gamma^n(V_{x,-r_1})=V_{\gamma^n(x), \gamma^n(-r_1)}$, and $\gamma^n([-r_1, x])$ converges to $\xi_+$ with respect to the cone topology of $\Delta \cup \partial \Delta$  when $n \to \infty$, one concludes there is indeed $N>0$ such that $\gamma^n([-r_1, x]) \subset (y, \xi_+)$, for every $n \geq N$. The conclusion follows from Lemma \ref{lem::lem11}.
\end{proof}

\begin{corollary}
\label{cor::easy_dinamics}
Let $\Delta$ be a locally finite Euclidean building,  $(W,S)$ the associated irreducible spherical Coxeter system, $I \subsetneq S$, and $\gamma \in \Aut(\Delta)$ a type-preserving, strongly $I$-regular hyperbolic element with attracting and repelling endpoints $\xi_+, \xi_-$ contained, respectively, in the interior of the minimal ideal opposite simplices $\sigma_+,\sigma_-$. Suppose there is an apartment $\mathcal{A}$ of $\Delta$ such that $\mathcal{A} \subset \Min(\gamma)$ and with $\sigma_+, \sigma_- \subset \partial_\infty \mathcal{A}$.

If $V_{\sigma_-}\subset \Opp(\sigma_+)$ is an open neighborhood of $\sigma_-$ with respect to the cone topology on $I_{\sigma_-}(\partial_\infty \Delta)$, then $\lim\limits_{n \to \infty} \gamma^{n}(V_{\sigma_-})=\Opp(\sigma_+)$, where $I_{\sigma_-}$ is the type of $\sigma_-$.
\end{corollary}
\begin{proof}
Without loss of generality, we can suppose that $V_{\sigma_-}$ is of the form $V_{x,-r}$, for some $x \in \ell$, and $-r \in (x,\xi_-)$. Notice that such $V_{x,-r}$ is an open neighborhood of $\xi_-$, and thus of $\sigma_-$, with respect to the cone topology on $I_{\sigma_-}(\partial_\infty \Delta)$. Then the conclusion follows from Lemma \ref{lem::lem12}.
\end{proof}

\begin{proposition}
\label{prop::transit_sigma}
Let $\Delta$ be a locally finite Euclidean building, $(W,S)$ the associated irreducible spherical Coxeter system, and $I \subsetneq S$. Let $\sigma_+,\sigma_-$ be two opposite ideal simplices with $\sigma_+$ of type $I$. Let $\{\gamma_n\}_{n\in \NN}$ be a sequence of strongly $I$-regular hyperbolic elements of $\Delta$ such that $\sigma_+,\sigma_-$ are the attracting and repelling minimal ideal simplices of $\gamma_n$, for every $n\in \NN$. Suppose there is an apartment $\mathcal{A}$ of $\Delta$ such that $\mathcal{A} \subset \Min(\gamma_n)$, for every $n\in \NN$, with $\sigma_+, \sigma_- \subset \partial_\infty \mathcal{A}$. Assume that $\lim\limits_{n\to \infty}d(x,\gamma_n(x))=\infty$ for some (hence every) point $x \in \mathcal{A}$.

If $V_{\sigma_-}\subset \Opp(\sigma_+)$ is an open neighborhood of $\sigma_-$ with respect to the cone topology on $I_{\sigma_-}(\partial_\infty \Delta)$, then $\lim\limits_{n \to \infty} \gamma_{n}(V_{\sigma_-})=\Opp(\sigma_+)$.
 \end{proposition}

\begin{proof}
Fix a point $x \in \mathcal{A}$ and a bi-infinite geodesic line $\ell \subset \mathcal{A}$ such that $x\in \ell$, and the two endpoints of $\ell$ are in the interior of  $\sigma_+,\sigma_-$, respectively. Then without loss of generality, we can suppose that $V_{\sigma_-}$ is of the form $V_{x,-r}$, for some $x \in \ell$, and $-r \in (x,\xi_-)$. Notice that such $V_{x,-r}$ is an open neighborhood of $\xi_-$, and thus of $\sigma_-$, with respect to the cone topology on $I_{\sigma_-}(\partial_\infty \Delta)$. 

Since $\ell$ is strongly $I$-regular in $\mathcal{A}$, and $\mathcal{A} \subset \Min(\gamma_n)$ for any $n$, we have that $\gamma_n(\ell)$ is still a strongly $I$-regular bi-infinite geodesic line in $\mathcal{A}$, which is parallel to $\ell$. Moreover, since $x\in \mathcal{A}$, $\lim\limits_{n\to \infty}d(x,\gamma_n(x))=\infty$, and $\{\gamma_n\}_{n\in \NN}$ is a sequence of strongly $I$-regular hyperbolic elements with the above mentioned properties, we get that $\ell \cap \gamma_n(Q(x, \sigma_-, \mathcal{A}))$ is larger and larger when $n\to \infty$. Recall that $Q(x, \sigma_-, \mathcal{A})$ is the $\sigma_-$ cone in $\mathcal{A}$ with base point $x$. In particular, $\lim\limits_{n\to \infty}(\ell \cap \gamma_n(Q(x, \sigma_-, \mathcal{A})) = \ell$.  This will easily imply that $\lim\limits_{n \to \infty}\gamma_{n}(V_{x,-r})= \Opp(\xi_+)$, which is equivalent to $\lim\limits_{n \to \infty} \gamma_{n}(V_{\sigma_-})=\Opp(\sigma_+)$.
\end{proof}

%

\subsection{Achieving the main goal}
\label{subsec::main_goal}


Let $\Delta$ be a locally finite thick Euclidean building, and let $G$ be a group of type-preserving automorphisms of $\Delta$, with a strongly transitive action on $\Delta$, and thus on $\partial_\infty \Delta$ as well. Fix two opposite Borel subgroups $B^{+}, B^{-}$ of $G$, having their associated ideal chambers in $\partial \Delta$ denoted by $c_+$ and $c_-$,  respectively. The ideal chambers $c_+,c_-$ are opposite. Take $T:= B^{-}\cap B^{+}$ and notice that $T$ fixes pointwise the ideal chambers $c_+, c_-$, and thus the unique ideal apartment in $\partial \Delta$ containing $c_+,c_-$. In particular, $T$ stabilizes setwise the unique apartment $\mathcal{A}$ of $\Delta$ having $c_+,c_-$ in its ideal boundary $\partial_\infty \mathcal{A}$. For a sub-simplex $\sigma_+$ of $c_+$, we denote its type by $I_{\sigma_+}$, and by $T_{\sigma_\pm}$ the set of all strongly $I_{\sigma_+}$-regular hyperbolic elements in $T$. By Theorem \ref{thm:ExistenceStronglyReg_1}, $G$ admits strongly $I_{\sigma_+}$-regular hyperbolic elements, but those might not be in $T_{\sigma_\pm} \subset T$. For the rest of the section we will work under the assumption that $T_{\sigma_\pm}\neq \emptyset$. This is the case when $G$ is for example a  split simple algebraic group over a non-Archimedean local field. Notice that $T_{\sigma_\pm}$ is also the set of all strongly $I_{\sigma_-}$-regular hyperbolic elements in $T$, where $I_{\sigma_-}$ is the type of the sub-simplex $\sigma_-$ of $c_-$ that is opposite $\sigma_+$.

\medskip
Let now $H$ be a closed subgroup of $G$ and $\sigma_+$ a sub-simplex of $c_+$ with opposite $\sigma_-$ in $c_-$.  We will consider the following assumptions on $H$ and $\sigma_\pm$:
\begin{enumerate}
\item[(OP$_{\sigma_-}$)]
 $HG_{\sigma_-}$ is open in $G$
\item[(NRP$_{\sigma_+}$)]
If  $\{a_n\}_{n \in \NN}$ is a sequence of $T_{\sigma_\pm}$ such that every $a_n$ has $\sigma_+$ as its attracting simplex,  $\lim\limits_{n\to \infty} \vert a_n \vert = \infty$, and if $\{a_n H a_n^{-1}\}_{n\in \NN}$ converges to $L$ in the Chabauty topology on $\mathcal{S}(G)$, then $L$ is a subgroup of the parabolic subgroup $P_{\sigma_+}:=G_{\sigma_+}$, and $L$ is normalized by an element of $T_{\sigma_\pm}$. In particular, $L$ fixes pointwise the simplex $\sigma_+$. 
\end{enumerate}

\begin{lemma}
\label{lem::open_H_chamber}
Let $\Delta$ be a locally finite thick Euclidean building, $G$ a group of type-preserving automorphisms of $\Delta$ with a strongly transitive action on $\Delta$, and $G_{\sigma_-}$ the parabolic subgroup of $G$ corresponding to the ideal chamber $\sigma_- \subset \partial_\infty \Delta$. Let $H$ be a closed subgroup of $G$ with property (OP$_{\sigma_-}$). 

Then there is an open neighborhood $V_{\sigma_-}$ of $\sigma_-$ in $I_{\sigma_-}(\partial_\infty \Delta)$, with respect to the cone topology on $\partial_\infty \Delta$, such that $H$ acts transitivley on the set $V_{\sigma_-}$, i.e. $H \sigma_- \cap V_{\sigma_-} = V_{\sigma_-}$.
\end{lemma}
\begin{proof}
Recall that the canonical projection map $G \to G/P_{\sigma_-}$ is continuous and open, and that $G/P_{\sigma_-}$ is homeomorphic to $I_{\sigma_-}(\partial_\infty \Delta) $ by the strongly transitive action of $G$ on $\Delta\cup \partial_\infty\Delta $. Also recall that $I_{\sigma_-}(\partial_\infty \Delta)$ is endowed with the cone topology. Since $HP_{\sigma_-}$ is open in $G$, we can choose a standard open neighborhood $V_{\sigma_-} \subset  I_{\sigma_-}(\partial_\infty \Delta)$ of $\sigma_-$ that is contained in $HP_{\sigma_-}/P_{\sigma_-}$, and such that $H$ acts transitively on $V_{\sigma_-}$, i.e. $H\sigma_- \cap V_{\sigma_-} =V_{\sigma_-}$. 
\end{proof}

Notice that Lemma \ref{lem::open_H_chamber} remains true if we do not assume that $G$ is strongly transitive on $\Delta$ but we have to change the property (OP$_{\sigma_-}$) by asking that  $H\sigma_- \subset I_{\sigma_-}(\partial_\infty \Delta)$ contains an open neighborhood of $\sigma_-$ with respect to $\partial_\infty \Delta$ and the cone topology.

As a consequence of Lemma \ref{lem::open_H_chamber} we get the following.
\begin{remark}
If $H$ is a closed subgroup of $G$ with property (OP$_{\sigma_-}$), then $H$ cannot be a subgroup of the parabolic $P_{\sigma'}$ of any of the sub-simplices $\sigma'$ of  the ideal simplex $\sigma_-$, as otherwise $H\sigma_-$ will not be open in $ I_{\sigma_-}(\partial_\infty \Delta)$ with respect to the cone topology.
\end{remark}

Finally, the main goal of this article is the next theorem. First recall the definition for contractions subgroups $U_{\gamma}^{+}$ from (\ref{equ::contr_subgroup}) given in Section \ref{subsec::para_cont_unit}.

Secondly, let us comment a little on the above two imposed assumptions. The open condition (OP$_{\sigma_-}$) appears naturally when $H$ is the fixed point group of an involution of a connected reductive group $G$ over a non-Archimedean local field. The same will be true for the very unnaturally looking condition (NRP$_{\sigma_+}$) (see \cite{HW93}). Now, in order to prove our main theorem in a very general setting, and to achieve the goal of this article, we need to impose a third property that will be explained in the next remark.

\begin{remark}
\label{rem::third_prop_conv}
Again let $H$ be a closed subgroup of $G$ and $\sigma_+$ a sub-simplex of $c_+$ with opposite $\sigma_-$ in $c_-$.  Assume $T_{\sigma_\pm}\neq \emptyset$, and that for $H$ and a sequence $\{a_n\}_{n \in \NN}$ of $T_{\sigma_\pm}$ the properties (OP$_{\sigma_-}$) and (NRP$_{\sigma_+}$) hold true. Notice that every element of $T_{\sigma_\pm}$ stabilizes the apartment $\mathcal{A}$ and acts on it as a translation, thus $\mathcal{A} \subset \Min(\alpha)$, for any $\alpha \in T_{\sigma_\pm}$.

By Lemma \ref{lem::open_H_chamber} there is a standard open neighborhood $V_{\sigma_-}$ of $\sigma_-$ in $I_{\sigma_-}(\partial_\infty \Delta)$, with respect to the cone topology on $\partial_\infty \Delta$, such that $H$ acts transitivley on the set $V_{\sigma_-}$.  Then, we apply Proposition \ref{prop::transit_sigma} to $V_{\sigma_-}$ and the sequence $\{a_n\}_{n \in \NN}$ of $T_{\sigma_\pm}$, and get that $ \lim\limits_{n \to \infty} a_{n}(V_{\sigma_-})=\Opp(\sigma_+)$.  In particular, since $a_{n}(V_{\sigma_-}) \subset a_n H a_n^{-1}\sigma_-$ is an open neighbohood of $c_-$ on which $a_n H a_n^{-1}$ acts transitively, we get the following. For every given $\sigma' \in \Opp(\sigma_+)$, there exists $N>0$ such that $\sigma' \in a_{n}(V_{\sigma_-})$, for any $n \geq N$. Since $V_{\sigma_-}$ is a standard open neighborhood, so is $a_{n}(V_{\sigma_-})$, with respect to a specific base point in $\Delta$ that depends on $a_n$. So by the transitivity of $ a_n H a_n^{-1}$ on $a_{n}(V_{\sigma_-})$, we can find a sequence $\{g_n \in a_n H a_n^{-1}\}_{n\geq N}$ such that $g_n(\sigma_-)=\sigma'$. Still, it might be the case that there is no convergent subsequence of $\{g_{n}\}_{n\geq \NN}$ in $G$, and so there will be no element in $L$ sending $\sigma_-$ to $\sigma'$.

Let $H_{\sigma_-}:=\{h\in H \; \vert \; h(\sigma_-)=\sigma_- \text{ pointwise}\}$ be the pointwise stabilizer in $H$ of $\sigma_-$. Under the above notation and the assumptions (OP$_{\sigma_-}$) and (NRP$_{\sigma_+}$), we further assume that
\begin{enumerate}
\item[(TranP$_{\sigma_-}$)] For every given $\sigma' \in \Opp(\sigma_+)$ the sequence of subsets $\{g_n a_n H_{\sigma_{-}} a_n^{-1}\}_{n\geq N}$ admits at least one strictly increasing convergent subsequence $\{ p_{n_k} \in g_{n_k} a_{n_k} H_{\sigma_{-}} a_{n_k}^{-1} \}_{n_k}$ in $G$.
\end{enumerate}
\end{remark}

\begin{theorem}
\label{thm::main_goal}
Let $\Delta$ be a locally finite thick Euclidean building, $G$ a closed group of type-preserving automorphisms of $\Delta$ with a strongly transitive action on $\Delta$, and $\sigma_-$ a sub-simplex of $c_- \subset \partial_\infty \Delta$. Let $\sigma_+$ be the sub-simplex of the ideal chamber $c_+$ that is opposite $\sigma_-$, and $H$ a closed subgroup of $G$ with property (OP$_{\sigma_-}$).  Assume that $T_{\sigma_\pm}\neq \emptyset$.

Then for any sequence $\{a_n\}_{n \in \NN}$ of $T_{\sigma_\pm}$ with properties (NRP$_{\sigma_+}$) and (TranP$_{\sigma_-}$) the corresponding Chabauty limit $L$ of $H$ with respect to  $\{a_n\}_{n \in \NN}$ admits the decomposition 
$$ L = U_{\gamma}^{+}(L) L_{\sigma_-},$$
where $\gamma \in T_{\sigma_\pm}$ is from (NRP$_{\sigma_+}$) and normalizes $L$, $U_{\gamma}^{+}(L): = L \cap U_{\gamma}^{+}$ is normal in $L$ with a transitive action on $\Opp(\sigma_+)$, and $L_{\sigma_-}:= L \cap G_{\sigma_+,\sigma_-}$. If moreover $U_{\gamma}^{+}$ acts simple-transitively on $\Opp(\sigma_+)$, then  $U_{\gamma}^{+}(L)= U_{\gamma}^{+}$ and $L = U_{\gamma}^{+} \rtimes L_{\sigma_-}$. 
\end{theorem}

\begin{proof}
Agian notice that every element of $T_{\sigma_\pm}$ stabilizes the apartment $\mathcal{A}$ and acts on it as a translation, thus $\mathcal{A} \subset \Min(\alpha)$, for any $\alpha \in T_{\sigma_\pm}$. 

Consider a sequence $\{a_n\}_{n \in \NN}$ of $T_{\sigma_\pm}$ with properties (NRP$_{\sigma_+}$) and (TranP$_{\sigma_-}$). Then, for every given $\sigma' \in \Opp(\sigma_+)$ we get at least one convergent subsequence $\{p_{n_k}\}_{n_k}$ whose limit in $G$ will be in $L$ by the definition of the Chabauty topology, and will send $\sigma_-$ to $\sigma'$. So  $L$ acts transitively on $\Opp(\sigma_+)$.
  
By property  (NRP$_{\sigma_+}$) we know that $L\leq P_{\sigma_+}:=G_{\sigma_+}$ and $L$ is normalized by an element $\gamma$ of $T_{\sigma_\pm}$. Recall from Proposition \ref{prop::geom_levi_decom} and Corollary \ref{coro::geom_levi_decom} that we get that 
$$P^{+}_{\gamma} = G_{\sigma_+}= P_{\sigma_+}= U_{\gamma}^{+} M_{\gamma},$$
where $M_{\gamma}:= P^{+}_{\gamma} \cap P^{-}_{\gamma}=G_{\sigma_+,\sigma_-}$, and $U_{\gamma}^{+}$ is normal in $G_{\sigma_+}$.

We claim we can apply \cite[Theorem 3.8 and Corollary 3.17]{BaWil} and obtain the wanted decomposition $L = U_{\gamma}^{+}(L) L_{\sigma_-}$. Indeed, we know that $\gamma \in T_{\sigma_\pm}$ is a strongly $I_{\sigma_+}$-regular hyperbolic element that normalizes $L$. Then, since $\gamma \in G_{\sigma_+, \sigma_-}$ we get that $L_{\sigma_-} = L\cap G_{\sigma_+, \sigma_-} = (\gamma L \gamma^{-1}) \cap (\gamma G_{\sigma_+, \sigma_-} \gamma^{-1}) =  \gamma L_{\sigma_-} \gamma^{-1}$, so $L_{\sigma_-}$ is $\gamma$-stable. Since $\gamma$ is an automorphism of $L$ as well, we are able to apply \cite[Theorem 3.8]{BaWil} to $G:=L$, $H:=L_{\sigma_-}$, and $U_{\alpha}:= L \cap U_{\gamma}^{+}= U_{\gamma}^{+}(L)$, which is easy to see is a normal subgroup of $L$ since $L \leq  P^{+}_{\gamma} = G_{\sigma_+}$. Therefore, we get that $U_{\gamma}^{+}(L)_{/ L_{\sigma_-}}=  U_{\gamma}^{+}(L) L_{\sigma_-}$ (here we import the notation from \cite{BaWil}).  Also, one can notice that by the definitions we have that $U_{\gamma}^{+}(L) L_{\sigma_-} \leq L$. Thus, as in the proof of \cite[Corollary 3.17]{BaWil}, it is enought to show that $L \subset U_{\gamma}^{+}(L)_{/L_{\sigma_-}}$. For completeness, we recall the proof here. Take $v \in L$, and since $L \leq G_{\sigma_+}=P^{+}_{\gamma}$ we have that $\{\gamma^{-n} v \gamma ^{n} \; \vert \; n\in \NN\} \subset L$ is bounded. Let $w$ be an accumulation point of that set, so $w$ is a limit of a subsequence in $\{\gamma^{-n} v \gamma ^{n}\}_{n\in \NN}$, in particular $w \in L$ as $L$ is closed and normalized by $\gamma$.  Now for every $m \in \ZZ$ we also have that $\gamma^{m} w \gamma^{-m}$ is again some accumulation point of $\{\gamma^{-n} v \gamma ^{n}\}_{n\in \NN}$. In particular, we have that the set $\{\gamma^{m} w \gamma^{-m} \; \vert \; m\in \ZZ\}$ is contained in the compact set $\overline{\{\gamma^{-n} v \gamma ^{n} \; \vert \; n \in \NN\}}$, thus bounded as well. This implies that $w \in L_{\sigma_-} = L \cap G_{\sigma_+} \cap G_{\sigma_-}$ which is a closed subgroup of $L$. Apply \cite[Lemma 3.9 part 3]{BaWil} to conclude that $v \in  U_{\gamma}^{+}(L)_{/ L_{\sigma_-}}$ as wanted.

Moreover, by the transitivity of $L$ on $\Opp(\sigma_+)$, the transitivity of $U_{\gamma}^{+}(L)$ on $\Opp(\sigma_+)$ follows as well. From here the conclusion of the theorem follows easily.
\end{proof}

%
%
%

\end{document}